\documentclass[aos,authoryear]{imsart}

\RequirePackage{amsthm,amsmath,amsfonts,amssymb}
\RequirePackage[numbers]{natbib}
\RequirePackage[colorlinks,citecolor=blue,urlcolor=blue]{hyperref}
\RequirePackage{graphicx}

\usepackage{amssymb}

\usepackage{dsfont}

\usepackage{mathabx}

\usepackage{graphicx}

\usepackage{hyperref}

\usepackage{xcolor}

\usepackage[autostyle]{csquotes}  

\usepackage{prodint}

\usepackage{enumitem}

\startlocaldefs
\theoremstyle{plain}

\newtheorem{theorem}{Theorem}[section]
\newtheorem{lemma}[theorem]{Lemma}
\theoremstyle{remark}
\newtheorem{definition}[theorem]{Definition}
\newtheorem{proposition}[theorem]{Proposition}

\newtheorem{assumption}{Assumption}
\newtheorem*{example}{Example}
\newtheorem{remark}[theorem]{Remark}

\newcommand\independent{\protect\mathpalette{\protect\independenT}{\perp}}
\def\independenT#1#2{\mathrel{\rlap{$#1#2$}\mkern2mu{#1#2}}}

\usepackage{mathbbol}

\DeclareSymbolFontAlphabet{\amsmathbb}{AMSb}%

\endlocaldefs

\begin{document}

\begin{frontmatter}
\title{Conditional Aalen--Johansen estimation}
\runtitle{Conditional Aalen--Johansen estimation}

\begin{aug}
\author[A]{\fnms{Martin}~\snm{Bladt}\ead[label=e1]{martinbladt@math.ku.dk}}
\and
\author[A]{\fnms{Christian}~\snm{Furrer}\ead[label=e2]{furrer@math.ku.dk}}
\address[A]{Department of Mathematical Sciences, University of Copenhagen \\
Universitetsparken 5, DK-2100 Copenhagen, Denmark \\[2mm] \printead[presep={}]{e1,e2}}

\end{aug}

\begin{abstract}
The conditional Aalen--Johansen estimator, a general-purpose non-parametric estimator of conditional state occupation probabilities, is introduced. The estimator is applicable for any finite-state jump process and supports conditioning on external as well as internal covariate information. The conditioning feature permits for a much more detailed analysis of the distributional characteristics of the process. The estimator reduces to the conditional Kaplan--Meier estimator in the special case of a survival model and also englobes other, more recent, landmark estimators when covariates are discrete. Strong uniform consistency and asymptotic normality are established under lax moment conditions on the multivariate counting process, allowing in particular for an unbounded number of transitions.
\end{abstract}



\end{frontmatter}

\section{Introduction}

Jump processes are a fundamental class of stochastic processes that have been widely used in statistics to model dynamical systems evolving over time, such as population dynamics and credit risk and, more broadly, in survival and event history analysis~\citep{Hoem1972,JarrowLandoTurnbull1997,AndersenBorganGillKeiding1993}. Markov jump processes on finite state spaces have received by far the most attention due their flexibility and tractability. Non-parametric estimation of their hazards and transition probabilities is available through the celebrated Nelson--Aalen and Aalen--Johansen estimators~\citep{Aalen1978,AalenJohansen1978}. Although parametric and semi-parametric models are more popular for quantifying distributional characteristics, the Aalen--Johansen estimator remains the benchmark tool for, among other things, visualizing state occupation probabilities.

However, in most applications, the Markov assumption is violated. Another important class of stochastic processes that have received widespread attention are semi-Markov jump processes~\citep{KorolyukBrodiTurbin1975,JanssenLimnios1999}. These models take into account the duration of time spent in the current state since entry, which can have important implications for the dynamics of the system being modeled. For example, in many real-world systems, the length of time spent in a particular state may be influenced by external factors or internal dynamics, which in turn can affect the probability of transitioning to a new state. Incorporating duration effects, or other forms of dependence that violate the Markov assumption, may therefore provide a more accurate and nuanced representation of the underlying system.

Semi-Markov jump processes have been used extensively in the analysis of a wide range of systems, including queuing systems~\citep{Fabens1961}, and in reliability, credit risk, and epidemiology modeling~\citep{LimniosOprisan2001,VasileiouVassiliou2006,Commenges1999}. The past trajectory of the process also plays an important role in biostatistics and related fields, where the duration of a treatment or disease progression may be critical, confer for instance with~\cite{Scheike2001},~\cite{Janssen1966}, and~\cite{Hoem1972}.

Although non-parametric estimation of homogeneous semi-Markov models is well-understood~\citep{LagakosSommerZelen1978}, the inhomogeneous case is more involved. This has led to different approaches, of which we highlight random time changes for progressive models according to~\cite{VoelkelCrowley1984}. {More recently, so-called landmark estimation~\citep{Houwelingen2007,PutterSpitoni2018}}, which applies to general non-Markovian models, has gained significant traction. Here the basic idea is to sub-sample on a discrete event of interest, typically the state of the process at an earlier point in time, and then construct classic estimators for the transition probabilities. Consistency of such estimators has been partly explored. For instance, weak uniform consistency using martingale techniques was established in~\cite{Beyersmann2023}. Under a stringent assumption of a bounded number of transitions, an asymptotic normality result is provided in~\cite{Glidden2002}.

An alternative to landmarking is to shift focus away from the occupation probabilities and target the infinitesimal distributional characteristics directly, for which kernel-based methods have played an instrumental role, for instance in~\cite{McKeagueUtikal1990},~\cite{NielsenLinton1995}, and~\cite{Dabrowska1997}. This strain of literature targets the compensator of the counting process by using kernels to sub-sample dynamically. A rather important limitation of this approach is that the temporal dependence structure of the jump process has to be specified in advance, taking away from a fully non-parametric method. Of course, if one is willing to adopt a particular dependence structure, these methods offer favorable estimation properties based on martingale techniques. A more subtle drawback is that computation of transition and occupation probabilities may become rather involved, depending on the dependence structure.

{
Using elements from both the kernel-based methods and landmarking, this paper proposes a non-parametric estimator of state occupation probabilities of an arbitrary jump process, while also allowing for the inclusion of covariates and right-censoring. We refer to this estimator as the conditional Aalen--Johansen estimator. As a means to obtain such an estimator, we also introduce the conditional Nelson--Aalen estimator, retaining the sound principle that the (conditional) Aalen--Johansen estimator is the product integral of the (conditional) Nelson--Aalen estimator. The latter estimator targets a certain matrix function, which can be thought of as an averaged compensator, and which in the Markov case reduces to the usual hazard matrix. This intermediate step, which is also found in landmark estimation, sets our approach apart from others based on martingale methods; it may be conceptually artificial, but it is mathematically effective.

To accommodate for conditioning, the usual empirical counts and at-risk indicators are replaced by estimators of Nadaraya--Watson type, and the conditional Aalen--Johansen estimator reduces to the conditional Kaplan--Meier (Beran) estimator~\citep{Beran1981,Dabrowska1989} in case of a survival model. To obtain estimators of transition probabilities or other dynamical quantities that may arise in, for instance, inhomogenous semi-Markov models, one can condition on internal covariates. Our specification allows for both continuous and discrete conditioning, which opens up the possibility of conditioning on events that directly depend on time or duration. These applications usually only require uni- or bi-dimensional covariate spaces, for which kernel-based methods do not suffer from their well-known weakness: the curse of dimensionality. Naturally, when conditioning on discrete random variables, the estimator is reduced to the landmark Aalen--Johansen estimator.}

Apart from the additional flexibility of the estimator, the main theoretical contributions of the paper are as follows. We provide unifying proofs of strong uniform consistency of our estimators. We also establish asymptotic normality using rather general moment conditions on the associated multivariate counting process (in particular, allowing for an unbounded number of transitions). The results are established using the theory of empirical processes, and they have as special cases many of the previous, empirical process-based, results in the literature. It is unclear whether martingale techniques offer an alternative route, though the cost of using classic empirical process theory is that the asymptotics are not established over the `maximal interval' of the censoring distribution, but rather only up to compact subsets.

The conditional Nelson--Aalen and Aalen--Johansen estimators are implemented in the R package~\cite{BladtFurrer2023}. In the accompanying Supplementary Material, we provide a simulation study, where the jump process follows an inhomogeneous semi-Markov model. Here, the practical versatility of the conditional Aalen--Johansen estimator is showcased, in particular its ability to recover probabilities that depend on the duration spent in the current state.

The rest of the paper is organized as follows. In Section~\ref{sec:estimator}, we introduce the conditional Nelson--Aalen and Aalen--Johansen estimators. Uniform consistency is proven in Section~\ref{sec:consistency}, while Section~\ref{sec:normality} concerns asymptotic normality. Subsection~\ref{subsec:prelim}, which is devoted to the finite dimensional distributions, is followed by Subsection~\ref{subsec:weak}, where weak convergence is established, and Subsection~\ref{subsec:plugin}, which is devoted to plug-in estimation for the covariances. Proofs may be found in Appendices~\ref{appA} through~\ref{appC}, while Appendix~\ref{ap:technical} concerns the bracketing central limit theorem.

\section{Conditional Aalen--Johansen estimator}\label{sec:estimator}

In this section, we introduce necessary notation and define the {quantities} of interest, namely the conditional occupation probabilities, as well as their kernel-based estimators. To accommodate right-censoring, we cast the occupation probabilities as a product integral of a certain auxiliary matrix function, which in the Markov case reduces to the usual hazard matrix function; this is in accordance with the landmarking literature. {We stress that the auxiliary matrix function at best may be understood as an averaged compensator of the counting process, but otherwise has no direct interpretation.} Estimators of the auxiliary matrix function and the occupation probabilities are constructed, adopting the smoothing approach of~\cite{NielsenLinton1995} to the present setup.

\subsection{Setting}\label{subsec:estimator_setting}

Let $Z=(Z_t)_{t\geq0}$ be a non-explosive pure jump process on a finite state space $\mathcal{Z}$, and take $(\Omega,\mathcal{F},\amsmathbb{P})$ to be the underlying probability space. For ease of notation, we require $\mathcal{Z} \subset \amsmathbb{R}$. Denote by $\tau$ the possibly infinite absorption time of $Z$.

Let $X$ be a random variable with values in $\amsmathbb{R}^d$ equipped with the Borel $\sigma$-algebra. We assume that the distribution of $X$ admits a density $g$ with respect to the Lebesgue measure $\lambda$. Later, in Section~\ref{sec:consistency}, we discuss the relaxation of absolute continuity to the case of known atoms.

We introduce a multivariate counting process $N$ with components $N_{jk}=(N_{jk}(t))_{t\geq0}$ given by
\begin{align*}
N_{jk}(t) = \# \{s \in (0,t] : Z_{s-} = j, Z_s = k\}, \quad t\geq0,
\end{align*}
for $j,k \in \mathcal{Z}$, $j\neq k$.

The following natural moment condition plays a major role in the sequel. It particularly implies the previously stated non-explosion of $Z$.
\begin{assumption}\label{ass:moment_N}
For some $\delta>0$, it holds that
\begin{align*}
\amsmathbb{E}\big[N_{jk}(t)^{1+\delta}\big] < \infty.
\end{align*}
\end{assumption}
Define the {conditional} occupation probabilities according to
\begin{align*}
p_j(t | x)
=
\amsmathbb{E}[\mathds{1}_{\{Z_t = j\}} \, | \, X = x].
\end{align*}
We denote by $p(t | x)$ the row vector with elements $p_j(t | x)$. We now introduce the \textit{cumulative conditional transition rates} and cast the conditional occupation probabilities as a product integral of these cumulative rates. To this end, let
\begin{align*}
p_{jk}(t | x)
=
\amsmathbb{E}[N_{jk}(t) \, | \, X = x],
\end{align*}
and define the cumulative conditional transition rates {as the matrix function $\Lambda(\cdot|x)$ with entries}
\begin{align*}
\Lambda_{jk}(t | x) = \int_0^t \frac{1}{p_j(s- | x)} p_{jk}(\mathrm{d}s | x){, \quad \Lambda_{jj}(t | x) = -\sum_{k:k\neq j}\Lambda_{jk}(t | x).}
\end{align*}
{Invoking the law of iterated expectations and employing heuristic notation, we see that
\begin{align*}
\Lambda_{jk}(\mathrm{d}t | x) = \amsmathbb{E}\big[\amsmathbb{E}\big[N_{jk}(\mathrm{d}t) \, \big| \, \mathcal{G}^{Z,X}_{t-}\big] \, \big| \, Z_{t-} = j, X = x\big],
\end{align*}
where $\mathcal{G}^{Z,X}_t := \sigma(X,Z_s : s \leq t)$. The inner quantity is the $(\mathcal{G},\amsmathbb{P})$-compensator of $N_{jk}$. Therefore, we may interpret the cumulative conditional transition rate as a sort of conditionally averaged compensator.

If $Z$ conditional on $X = x$ is Markov with suitably regular transition rates $\mu(\cdot | x)$, then
\begin{align*}
\amsmathbb{E}\big[N_{jk}(\mathrm{d}t) \, \big| \, \mathcal{G}^{Z,X}_{t-}\big] = \mathds{1}_{\{Z_{t-} = j\}} \mu(t \, | \, X) \, \mathrm{d}t
\end{align*}
and hence
\begin{align*}
\Lambda_{jk}(t | x) = \int_0^t \mu(s | x) \, \mathrm{d}s,
\end{align*}
which helps explain the terminology \textit{cumulative conditional transition rate}. In this case, the (conditional) transition probabilities uniquely solve Kolmogorov's forward differential equations or, equivalently, may be cast as the product-limit
\begin{align}\label{eq:trans_probs}
\amsmathbb{P}(Z_t = k \, | \, Z_s = j, X = x) = \bigg[\Prodi_s^t \big(\text{Id} + \Lambda(\mathrm{d}s | x)\big)\bigg]_{jk}.
\end{align}
If $Z$ is not Markov, then~\eqref{eq:trans_probs} is not attainable. However, if $\Lambda_{jk}(t {\, | \, x}) < \infty$ (which we shall assume throughout) it still holds that
\begin{align}\label{eq:prod_int_rep}
p(t|x) = p(0|x)  \Prodi_0^t \big(\text{Id} + \Lambda(\mathrm{d}s | x)\big).
\end{align}
In particular, estimation of transition probabilities must be based on conditining with internal covariates. The identity~\eqref{eq:prod_int_rep} has been established in~\cite{Overgaard2019b}, but subject to a somewhat artificial technical condition. It suffices, however, to assume that $\Lambda_{jk}(t {\, | \, x}) < \infty$, as we shall now demonstrate.

The assumption that $\Lambda$ is finite implies that the product integral
\begin{align*}
q(t|x) = \Prodi_0^t \big(\text{Id} + \Lambda(\mathrm{d}s | x)\big)
\end{align*}
is well-defined. Therefore, the forward integral equation
\begin{align*}
q(t|x)
&=
q(0|x)
+
\int_0^t q(s-|x) \Lambda(\mathrm{d}s)
\end{align*}
holds, which again implies that
\begin{align*}
p(0|x)q(t|x)
&=
p(0|x)q(0|x)
+
\int_0^t p(0|x)q(s-|x) \Lambda(\mathrm{d}s).
\end{align*}
However, from the identity
\begin{align*}
\mathds{1}_{\{Z_t = j\}} = \mathds{1}_{\{Z_0 = j\}} + \sum_{k \in \mathcal{Z}\atop k \neq j} \big( N_{kj}(t) - N_{jk}(t) \big)
\end{align*}
and the definition of $p$ and $\Lambda$ we also conclude that
\begin{align*}
p(t|x)
&=
p(0|x)
+
\int_0^t p(s-|x) \Lambda(\mathrm{d}s).
\end{align*}
Consequently, both $p(0|x)q(\cdot|x)$ and $p(\cdot|x)$ solve the same integral equation. To verify~\eqref{eq:prod_int_rep}, it thus suffices to establish the uniqueness of solutions to this class of integral equations. This, however, follows from a suitable application of the Banach fixed-point theorem, compare also with~\cite{BathkeChristiansen2024}.}

We are interested in the situation where the observation of $Z$ is right-censored. To this end we introduce a strictly positive random variable $R$ describing right-censoring, so that we actually observe the triplet
\begin{align*}
\big(X, (Z_t)_{0 \leq t \leq R}, \tau \wedge R\big).
\end{align*}
In the following, we fix $x\in\amsmathbb{R}^d$ and focus on a compact time interval $[0,\theta_x]$ for which the right endpoint is required to satisfy that there exists $\delta>0$ such that
\begin{align}\label{eq:right_endpoint}
\theta_x < \inf_{z \in \amsmathbb{R}^d : \lVert z - x \rVert \leq \delta, g(z)>0} \, \inf_{0\leq t < \infty} \{\amsmathbb{P}(R\leq t | X = z)=1 \}.
\end{align}
{That is, we choose $\theta_x$ such that locally around $x$ it is below the right endpoint of the supports of the conditional distributions of $R$ given $X$.}

Whenever the Markov assumption fails, the following assumption is key.
\begin{assumption}\label{ass:independence}
Right-censoring is conditionally entirely random:
\begin{align*}
R \underset{X}{\independent} Z.
\end{align*}
\end{assumption}
{In this paper, we focus on a single sampling effect, namely right-censoring. Left-truncation is also a frequently encountered phenomena, and although a general treatment is outside the scope of this paper, we would like to briefly comment on its inclusion. In the absence of covariates, it has been shown in~\cite{Beyersmann2023} that the assumption of entirely randomness is key. We would expect this to essentially remain the case in the present setting, with independence replaced by conditional independence given $X$.}

We introduce
\begin{align*}
p_j^{\texttt{c}}(t | x)
&=
\amsmathbb{E}[\mathds{1}_{\{Z_t = j\}}\mathds{1}_{\{t < R\}} \, | \, X = x], \\
p_{jk}^{\texttt{c}}(t | x)
&=
\amsmathbb{E}[N_{jk}(t \wedge R) \, | \, X = x].
\end{align*}
It follows from Assumption~\ref{ass:independence} that
\begin{align*}
p_j^{\texttt{c}}(t | x)
&=
\amsmathbb{P}(t < R \, | \, X = x) p_j(t|x), \\
p_{jk}^{\texttt{c}}(t | x)
&=
\int_0^t \amsmathbb{P}(s {\leq} R \, | \, X = x) \, p_{jk}(\mathrm{d}s | x),
\end{align*}
from which we may conclude that
\begin{align}\label{eq:Lambda_re}
\Lambda_{jk}(t | x) = \int_0^t \frac{1}{p_j^{\texttt{c}}(s- | x)} p_{jk}^{\texttt{c}}(\mathrm{d}s | x){.}
\end{align}
\begin{remark}
If $Z$ is Markov, then one may forego Assumption~\ref{ass:independence}. Instead, assuming the weaker condition that the compensator of the multivariate counting process remains unchanged upon addition of censoring information, confer also with Subsection~III.2.2 in~\cite{AndersenBorganGillKeiding1993}, we obtain the same conclusion:
\begin{align*}
\int_0^t \frac{1}{p_j^{\texttt{c}}(s- | x)} p_{jk}^{\texttt{c}}(\mathrm{d}s | x) 
&=
\int_0^t \frac{1}{p_j^{\texttt{c}}(s- | x)} \mathrm{d}\amsmathbb{E}[\mathds{1}_{\{s \leq R\}}\mathds{1}_{\{Z_{s-}=j\}} \Lambda_{jk}(s|x) \, | \, X = x] \\
&=
\int_0^t \frac{p_j^{\texttt{c}}(s- | x)}{p_j^{\texttt{c}}(s- | x)} \Lambda_{jk}(\mathrm{d}s|x) 
=
\Lambda_{jk}(t|x).
\end{align*}
\end{remark}

\subsection{Estimators}\label{subsec:estimator_estimator}

Consider i.i.d.\ replicates $\big(X^\ell, (Z_t^\ell)_{0 \leq t \leq R^\ell}, \tau^\ell \wedge R^\ell\big)_{\ell = 1}^n$. The aim now is to construct an estimator for the conditional occupation probabilities $p_j(t | x)$, but from \eqref{eq:prod_int_rep}, it essentially suffices to estimate the cumulative conditional transition rates $\Lambda_{jk}(t | x)$, which we do using Equation~\eqref{eq:Lambda_re}.

{Let $K_i$ be { kernel functions of bounded variation with bounded support (hence they are bounded)}, and let $(a_n)$ be a band sequence. We introduce the usual kernel estimator for the density $g$:
\begin{align*}
\mathbb{g}^{(n)}(x)
=
\frac{1}{n}
\sum_{\ell = 1}^n \mathbb{g}^{(n,\ell)}(x)
=
\frac{1}{n}
\sum_{\ell = 1}^n
\prod_{i=1}^d \frac{1}{a_n}  K_i\Big(\frac{x_i - X_i^\ell}{a_n}\Big).
\end{align*}
Based on the kernel estimator for $g$, we form the following Nadaraya--Watson type kernel estimators:
\begin{align*}
\amsmathbb{N}^{(n)}_{jk}(t | x)
&=
\frac{1}{n} \sum_{\ell = 1}^n \amsmathbb{N}^{(n,\ell)}_{jk}(t | x) = \frac{1}{n} \sum_{\ell=1}^n \frac{N^\ell_{jk}(t \wedge R^\ell) \mathbb{g}^{(n,\ell)}(x)}{\mathbb{g}^{(n)}(x)}, \\
\amsmathbb{I}_j^{(n)}(t | x)
&=
\frac{1}{n} \sum_{\ell = 1}^n \amsmathbb{I}_j^{(n,\ell)}(t | x) = \frac{1}{n} \sum_{\ell=1}^n \frac{\mathds{1}_{\{t < R^\ell\}} \mathds{1}_{\{Z^\ell_t = j\}} \mathbb{g}^{(n,\ell)}(x)}{\mathbb{g}^{(n)}(x)}.
\end{align*}
We are now ready to introduce the conditional Nelson--Aalen and Aalen--Johansen estimators.}
\begin{definition}
The \textit{conditional Nelson--Aalen estimator} for the cumulative conditional transition rates is defined as { the matrix function $\mathbb{\Lambda}^{(n)}(\cdot | x)$ with entries}
\begin{align}\label{eq:canonical_nelson_aalen}
\mathbb{\Lambda}_{jk}^{(n)}(t | x) = \int_0^t  \frac{1}{\amsmathbb{I}_j^{(n)}(s- | x)} \amsmathbb{N}_{jk}^{(n)}(\mathrm{d}s | x).
\end{align}
\end{definition}
\begin{definition}
The \textit{conditional Aalen--Johansen estimator} for the conditional occupation probabilities is defined according~\eqref{eq:prod_int_rep}, but with the cumulative conditional transition rates replaced by the conditional Nelson--Aalen estimator, that is
\begin{align*}
\mathbb{p}^{(n)}(t | x) = \mathbb{p}^{(n)}(0 | x) \Prodi_0^t \big(\text{Id} + \mathbb{\Lambda}^{(n)}(\mathrm{d}s | x)\big),
\end{align*}
where $\mathbb{p}_j^{(n)}(0 | x) = \amsmathbb{I}_j^{(n)}(0 | x)$.
\end{definition}
Even in case of no covariates, and under the assumption that $Z$ is Markovian, expressions such as {\eqref{eq:Lambda_re} and} \eqref{eq:canonical_nelson_aalen} are problematic, since one might be dividing by zero. It is well-understood that this issue cannot be directly resolved, but needs to be circumvented. Typically, this is done via assumptions on the local behavior of the jump process, confer with~\cite{Aalen1978} and Subsection~IV.1.2 in~\cite{AndersenBorganGillKeiding1993}. We adopt a different approach, changing instead the targets of estimation by means of perturbation. To this end, for $\varepsilon>0$ we consider the estimator
\begin{align*}
\mathbb{\Lambda}_{jk}^{(n,\varepsilon)}(t | x) = \int_0^t  \frac{1}{\amsmathbb{I}_j^{(n)}(s- | x) \vee \varepsilon} \amsmathbb{N}_{jk}^{(n)}(\mathrm{d}s | x),
\end{align*}
which targets the perturbed cumulative conditional transition rates given by
\begin{align*}
\Lambda_{jk}^{(\varepsilon)}(t | x) &= \int_0^t \frac{1}{p_j^{\texttt{c}}(s- | x) \vee \varepsilon}  p_{jk}^{\texttt{c}}(\mathrm{d}s | x) < \infty.
\end{align*}
In similar fashion, we consider the associated estimator
\begin{align*}
\mathbb{p}^{(n,\varepsilon)}(t | x)
=
\mathbb{p}^{(n)}(0 | x) \Prodi_0^t \big(\text{Id} + \mathbb{\Lambda}^{(n,\varepsilon )}(\mathrm{d}s | x)\big),
\end{align*}
which targets the perturbed conditional occupation probabilities given by
\begin{align*}
p^{(\varepsilon)}(t | x)
=
p(0 | x) \Prodi_0^t \big(\text{Id} + \Lambda^{(\varepsilon )}(\mathrm{d}s | x)\big).
\end{align*}
Observe that, with $0=t_0<t_1<t_2<\cdots$ a sequence of time points that describes all observed jump times in the sample,
\begin{align*}
&\mathbb{p}^{(n,\varepsilon)}_j(t_{\ell+1} |x)-\mathbb{p}^{(n,\varepsilon)}_j(t_{\ell} |x)\\
&=\sum_{\substack{k \in \mathcal{Z} \\ k \neq j}} \mathbb{p}^{(n,\varepsilon)}_k(t_{\ell} | x)\Delta\mathbb{\Lambda}_{kj}^{(n,\varepsilon)}(t_{\ell +1} | x) -\mathbb{p}^{(n,\varepsilon)}_j(t_{\ell} | x) \sum_{\substack{k \in \mathcal{Z} \\ k \neq j}} \Delta\mathbb{\Lambda}_{jk}^{(n,\varepsilon)}(t_{\ell +1} | x)
\end{align*}
and starting value $\mathbb{p}^{(n)}(0 | x) = \amsmathbb{I}^{(n)}(0 | x)$.

Whenever the non-zero values of the conditional occupation probabilities are bounded away from zero, any asymptotic results for the perturbed estimators and their targets carry over to the original estimators and their targets. This would, for instance, be guaranteed by Assumption~6.1 in~\cite{Aalen1978}. The results also carry over locally when boundedness away from zero is only guaranteed on certain subsets of the interior of the support of $R$. 

\begin{example}
{Suppose $\theta_x$ satisfies~\eqref{eq:right_endpoint}}, and suppose that the cumulative conditional transition rate $\Lambda_{jk}(\cdot | x)$ is piece-wise constant with discontinuities $0 = t_0 < t_1 < \cdots < t_M \leq {\theta_x}$. {This means that conditional on $X = x$ and between time zero and $\theta_x$, the process $Z$ may only jump at the specific time points $t_0,t_1,\ldots,t_M$ (discrete model).} Then for
\begin{align*}
\varepsilon = \amsmathbb{P}({\theta_x} < R \, | \, X = x) \hspace{-4mm} \prod_{i=1,\ldots,M\atop p_j(t_{i-1} \, | \, x)>0} \hspace{-4mm} p_j(t_{i-1} \, | \, x) > 0
\end{align*}
it holds that $\Lambda_{jk}^{(\varepsilon)}(t | x) = \Lambda_{jk}(t | x)$ on $[0,{\theta_x}]$.
\end{example}

\begin{example}
Set $\Lambda_{jj}(\cdot|x) = -\sum_{k: k \neq j} \Lambda_{jk}(\cdot|x)${, and suppose that this function is continuous; for the general case, see Lemma~8 in~\cite{Overgaard2019b}}. {Suppose $\theta_x$ satisfies~\eqref{eq:right_endpoint}}, and suppose that $p_j(t \, | \, x) > 0$ for some $0 \leq t < {\theta_x}$. Then $p_j(s \, | \, x) \geq p_j(t \, | \, x) e^{\Lambda_{jj}({\theta_x} | x)-\Lambda_{jj}(t | x)}$ for $t < s \leq {\theta_x}$, and hence for 
\begin{align*}
\varepsilon = p_j(t \, | \, x) e^{\Lambda_{jj}({\theta_x} | x)-\Lambda_{jj}(t | x)} \amsmathbb{P}({\theta_x} < R \, | \, X = x) > 0
\end{align*}
it holds that $\Lambda_{jk}^{(\varepsilon)}(s | x) - \Lambda_{jk}^{(\varepsilon)}(t | x) = \Lambda_{jk}(s | x) - \Lambda_{jk}(t | x)$ on $(t,{\theta_x}]$.
\end{example}

{
\begin{remark}
The above examples highlight that the choice of $\varepsilon$ is immaterial whenever the underlying probabilistic model has occupation probabilities far enough from zero, in which case we may actually take $\varepsilon=0$. In general, the underlying probabilistic model is unknown and hence a choice must be made. However, the estimators are always well defined for $\varepsilon=0$, and it is recommended to use this value, since it is only in the limiting case that $\varepsilon>0$ plays a role. 

In the classic absolutely continuous Markovian case, there is an alternative approach to handling division by zero based on assumptions tailored to the application of martingale theory, specifically Lenglart's inequality (for weak consistency) and Rebolledo's martingale central limit theorem (for asymptotic normality), confer with Subsection~IV.1.2.\ in~\cite{AndersenBorganGillKeiding1993}. Recently, in~\cite{Beyersmann2023}, this approach has been adapted to establish weak consistency even in the absence of Markovianity, but still assuming absolutely continuous jump times. In our situation it may also be possible to establish limit theorems, with the notable exception of strong consistency, along similar lines.
\end{remark}}

\begin{remark}
In general, inverse probability of censoring weighting offers an alternative approach when dealing with covariates and censoring, see~\cite{KoulSusarlaRyzin1981} and~\cite{LaanRobins2003}. In~\cite{MostajabiDatta2013} and~\cite{SiriwardhanaKulasekeraDatta2018}, inverse probability of censoring weighting has been used in simpler setup than ours, for instance considering only progressive jump processes. While inverse probability of censoring weighting allows for added dependence between the jump process and the censoring mechanism, it also requires the specification of a censoring model, which takes away from a fully non-parametric approach.
\end{remark}

{
In the remainder of the paper, we impose the following standard assumption, confer with~\cite{Stute1986}.
\begin{assumption}\label{ass:band_seq}
The band sequence satisfies:
\begin{enumerate}[label = (\alph*)]
\item $a_n \to 0$ as $n \to \infty$
\item $n a_n^d \to \infty$ as $n \to \infty$
\item $\sum_{n=1}^\infty c^r_n < \infty$ for some $r\in(1,1+\delta)$, where $c_n := (n a_n^d)^{-1} \log{n}$.
\end{enumerate}
\end{assumption}
For instance, one may for $(1+\delta)^{-1} < \eta < 1$ take
\begin{align*}
a_n^d \sim \frac{\log{n}}{n^{1-\eta}}.
\end{align*}
It is well-known that 
\begin{align*}
\mathbb{g}^{(n)}(x)\overset{\text{a.s.}}{\to} g(x), \quad n \to \infty,
\end{align*}
and, for any random variable $Y$ with i.i.d.\ replicates $(Y^\ell)$ satisfying $\amsmathbb{E}[|Y|^{1+\delta}] < \infty$, that
\begin{align}\label{eq:nada_conv}
\frac{1}{n}\frac{\sum_{\ell=1}^n Y^\ell\: \mathbb{g}^{(n,\ell)}(x)}{\mathbb{g}^{(n)}(x)} \overset{\text{a.s.}}{\to} \amsmathbb{E}[Y \, | \, X=x],\quad n \to \infty.
\end{align}
\begin{remark}\label{rmk:comp_point}
Due to Assumption~\ref{ass:moment_N}, it follows from~\eqref{eq:nada_conv} that
\begin{align}\label{eq:count_point}
\amsmathbb{N}^{(n)}_{jk}(t | x) \overset{\text{a.s.}}{\to} p_{jk}^{\texttt{c}}(t | x), \quad n \to \infty.
\end{align}
Furthermore, it holds that
\begin{align}\label{eq:indicator_point}
\amsmathbb{I}_j^{(n)}(t | x) \overset{\text{a.s.}}{\to} p_j^{\texttt{c}}(t | x), \quad n \to \infty.
\end{align}
\end{remark}
Due to the nature of the processes involved, we may actually strengthen the point-wise convergence from Remark~\ref{rmk:comp_point} to uniform convergence. The basic idea is first establish uniform consistency for $\amsmathbb{N}^{(n)}_{jk}$ via a Glinveko--Cantelli type argument, and then lift this convergence to first $\amsmathbb{I}^{(n)}_j$ and finally via standard arguments to $\mathbb{\Lambda}^{(n,\varepsilon)}$ and then $\mathbb{p}^{(n,\varepsilon)}$.}

\section{Uniform consistency}\label{sec:consistency}

Extending the consistency results from a survival setting to the multi-state setting comes with a variety of challenges. For instance, the counting process in survival analysis has a single jump, whereas our setting deals with a possibly unbounded number of jumps. The assumption of a bounded number of jumps is not desirable in this context. In particular, it excludes special canonical cases such as Markov and semi-Markov models. Another challenge in the multi-state setting is that exposure is not a monotone process, which requires us to rely heavily on an identity linking exposures to sums of counting processes.

\subsection{Results}\label{subsec:consistency_primary}

To establish uniform consistency of the estimators introduced in Subsection~\ref{subsec:estimator_estimator}, we first directly prove uniform consistency of the components. In a next step, we follow along the lines of Subsection~7.3 in~\cite{ShorackWellner2009} to establish uniform consistency also of the estimators of interest.

Recall {again} that for any jump process we may link sojourns and transitions through the identity
\begin{align*}
\mathds{1}_{\{Z_t = j\}} = \mathds{1}_{\{Z_0 = j\}} + \sum_{k \in \mathcal{Z}\atop k \neq j} \big( N_{kj}(t) - N_{jk}(t) \big).
\end{align*}
In similar fashion, through integration by parts, we may include the right-censoring mechanism and obtain
\begin{align}\label{eq:decomp}
\mathds{1}_{\{Z_t = j\}}\mathds{1}_{\{t < R\}} = \mathds{1}_{\{Z_0 = j\}} - \mathds{1}_{\{R \leq t\}} \mathds{1}_{\{Z_R = j\}} + \sum_{k \in \mathcal{Z}\atop k \neq j} \big( N_{kj}(t {\wedge R}) - N_{jk}(t {\wedge R}) \big){.}
\end{align}
From~\eqref{eq:decomp}, we may conclude that for the kernel estimators and the theoretical quantities, respectively, it holds that
\begin{align} \label{eq:indicator_decomp}
\amsmathbb{I}^{(n)}_j(t | x)
&=
\amsmathbb{I}^{(n)}_j(0 | x) - \amsmathbb{C}_j^{(n)}(t | x) + \sum_{k \in \mathcal{Z}\atop k \neq j} \big( \amsmathbb{N}^{(n)}_{kj}(t | x) - \amsmathbb{N}^{(n)}_{jk}(t | x) \big), \\  \label{eq:pc_decomp}
p_j^{\texttt{c}}(t | x)
&=
p_j^{\texttt{c}}(0 | x) - C_j(t|x) + \sum_{k \in \mathcal{Z}\atop k \neq j} \big(p_{kj}^{\texttt{c}}(t | x) - p_{jk}^{\texttt{c}}(t | x) \big),
\end{align}
where $C_j(t | x) := \amsmathbb{P}(R \leq t, Z_R = j \, | \, X = x)$ with Nadaraya--Watson type kernel estimator
\begin{align*}
\amsmathbb{C}_j^{(n)}(t | x) = \frac{1}{n} \sum_{\ell = 1}^n \frac{\mathds{1}_{\{R^\ell \leq t\}} \mathds{1}_{\{Z^\ell_{R^{\ell}} = j\}} \mathbb{g}^{(n,\ell)}(x)}{\mathbb{g}^{(n)}(x)}.
\end{align*}
Note that according to~\eqref{eq:nada_conv},
\begin{align}\label{eq:censoring_point}
\amsmathbb{C}_j^{(n)}(t | x) \overset{\text{a.s.}}{\to} C_j(t | x) = \amsmathbb{P}(R \leq t, Z_R = j \, | \, X = x), \quad n \to \infty.
\end{align}
{Equation~\eqref{eq:indicator_decomp} is of independent interest.} The fact that we may basically recover $\amsmathbb{I}^{(n)}_j$ from the kernel estimators of the counts is used to great extent in the proofs. It is, however, also directly relevant in connection with the implementation of the estimators, where it can lead to significant reductions in computational load. Indeed, construction of $\amsmathbb{N}^{(n)}_{jk}(\cdot | x)$ is relatively cheap compared to $\amsmathbb{I}^{(n)}_j(\cdot | x)$, and identity~\eqref{eq:indicator_decomp} serves to circumvent the explicit computation of the latter.

Notice that the correction factor $C_j(t | x)$ disappears as censoring becomes lighter, giving a direct interpretation of~\eqref{eq:pc_decomp} as expected inflow and outflow of a state, but adjusted for censoring. The correction factors are such that summation with respect to $j$ simply yields the distribution of the censoring mechanism.

{Consistency of even the Kaplan--Meier estimator on the `maximal interval' relies on either martingale techniques in $t$ (for convergence in probability) or reverse supermartingale techniques in $n$ (for almost sure convergence)~\citep{StuteWang1993,BakryGillMolchanov1994}.} In any case, our kernel disrupts these delicate properties. By conceding convergence only up to compact subsets within the support of the censoring variable $R$, we may concentrate on the extension to a multi-state setting, using empirical process theory.
\begin{proposition}\label{prop:N_I_const}
{Suppose $\theta_x$ satisfies~\eqref{eq:right_endpoint}}. It then holds that
\begin{align} \label{eq:count_uniform}
&\sup_{0 \leq t \leq {\theta_x}} \Big| \amsmathbb{N}^{(n)}_{jk}(t | x) - p_{jk}^{\texttt{c}}(t | x) \Big| \overset{\text{a.s.}}{\to} 0, &&n \to \infty, \\ \label{eq:indicator_uniform}
&\sup_{0 \leq t \leq {\theta_x}} \Big| \amsmathbb{I}^{(n)}_j(t | x) - p_j^{\texttt{c}}(t | x) \Big| \overset{\text{a.s.}}{\to} 0, &&n \to \infty. 
\end{align}
\end{proposition}

\begin{theorem}[Strong uniform consistency]\label{theo:L_P_const}
{Suppose $\theta_x$ satisfies~\eqref{eq:right_endpoint}}. It then holds that
\begin{align*}
&\sup_{0 \leq t \leq {\theta_x}} \big| \mathbb{\Lambda}_{jk}^{(n,\varepsilon)}(t | x) - \Lambda_{jk}^{(\varepsilon)}(t | x) \big| \overset{\text{a.s.}}{\to} 0, \quad n \to \infty. \\
&\sup_{0 \leq t \leq {\theta_x}} \big| \mathbb{p}_j^{(n,\varepsilon)}(t | x) - p_j^{(\varepsilon)}(t | x) \big| \overset{\text{a.s.}}{\to} 0, \quad n \to \infty.
\end{align*}
\end{theorem}

\subsection{Inclusion of atoms}\label{subsec:cosistency_atoms}

Theoretically, kernel functions work not just for absolutely continuous covariates, but in all cases. Sometimes one may, however, take particular interest in cases where some of the covariates have known atoms. Examples include the case when a covariate is the state of the jump process at an earlier point in time (landmarking) or when considering the duration since a sojourn in a particular state as a covariate. For the Markov case this is of course superfluous, but for general jump processes the inclusion of such internal covariates is necessary to get hold of the entire dynamics from and to any given points in time.

In this connection, we now propose a simple kernel which mixes sub-sampling with usual absolutely continuous kernel methods. We aim to embed landmarking methods as a particular case of our specification. The analysis of such an estimator is straightforward using standard techniques, though the practical implications are substantial when analyzing, for instance, semi-Markov models.

Thus, in the following, suppose that the distribution of $X$ is not absolutely continuous with respect to the Lebesgue measure but instead admits a density $g$ with respect to $\nu$ given by
\begin{align*}
\nu = \bigotimes_{i=1}^d \bigg( \lambda + \sum_{m=1}^\infty \delta_{\alpha_m} \bigg),
\end{align*}
where $\delta_y$ denotes the Dirac measure in $y$, and $\alpha = \{\alpha_1, \alpha_2, \ldots\}$ are the known potential atoms for the marginal distributions of $X$. The duration process associated with a semi-Markov model has, for instance, a known atom at the current point in time.

We propose the following kernel estimator for the density $g$:
\begin{align*}
\mathbb{g}^{(n)}(x)
&=
\frac{1}{n}
\sum_{\ell = 1}^n \mathbb{g}^{(n,\ell)}(x) \\
&=
\frac{1}{n}
\sum_{\ell = 1}^n
\prod_{i=1}^d \bigg( \frac{1}{a_n}  K_i\Big(\frac{x_i - X_i^\ell}{a_n}\Big) \mathds{1}_{\{X_i^\ell \notin \alpha \}} \mathds{1}_{\{x_i \notin \alpha\}} + \mathds{1}_{\{X_i^\ell = x_i\}} \mathds{1}_{\{x_i \in \alpha\}}\bigg).
\end{align*}
The consistency of the estimator under Assumption~\ref{ass:band_seq} follows from the usual case, but taking explicitly care of the known atoms.
\begin{proposition}\label{prop:kernel_atoms}
It holds that
\begin{align*}
\mathbb{g}^{(n)}(x)\overset{\text{a.s.}}{\to} g(x), \quad n \to \infty,
\end{align*}
and, for any random variable $Y$ with i.i.d.\ replicates $(Y^\ell)$ satisfying $\amsmathbb{E}[|Y|^{1+\delta}] < \infty$, that
\begin{align*}
\frac{1}{n}\frac{\sum_{\ell=1}^n Y^\ell\: \mathbb{g}^{(n,\ell)}(x)}{\mathbb{g}^{(n)}(x)} \overset{\text{a.s.}}{\to} \amsmathbb{E}[Y \, | \, X=x],\quad n \to \infty.
\end{align*}
\end{proposition}

\begin{remark}
Referring to the proof, for a fixed $x\in\amsmathbb{R}^d$ on the form $x=(x^a,x^c)$ with $x^c$ of dimension $d_2 = d - d_1$, we may relax b) and c) of Assumption~\ref{ass:band_seq} by replacing $d$ with $d_2$.
\end{remark}
The proof of Proposition~\ref{prop:kernel_atoms} is elementary and hinges on multiple reductions to the absolutely continuous case. In general, this technique can also be applied, together with the other arguments we provide, to prove uniform consistency and asymptotic normality of the estimators of interest. {Note that~\eqref{eq:right_endpoint} must then be adjusted for atomic $x$.} However, to keep the presentation clear and concise, in the remainder of the paper we turn our attention back to the absolutely continuous case. At no point would atom inclusion pose a methodological issue.

\section{Asymptotic normality}\label{sec:normality}

{This section is devoted to establishing asymptotic normality of our estimators. The proof technique is constructive and consists of first using the bracketing central limit theorem to obtain weak convergence of the empirical processes associated with $\amsmathbb{N}^{(n)}$ and $\amsmathbb{I}^{(n)}$, second applying the Skorokhod representation to extend the weak convergence to the empirical process of $\mathbb{\Lambda}^{(n)}$, and finally alluding to the functional $\delta$-method to obtain weak convergence of the empirical process associated with $\mathbb{p}^{(n)}$.

Subsection~\ref{subsec:prelim} deals with a preliminary result on the finite-dimensional distributions and is followed by Subsection~\ref{subsec:weak}, where we establish the main weak convergence results. Finally, Subsection~\ref{subsec:plugin} provides plug-in estimators for the covariances.}

\subsection{Preliminaries}\label{subsec:prelim}

In this subsection, we study in a first step the finite-dimensional asymptotics of the components of the estimators. The following assumptions on the local behavior around $x$ are routine and typically not considered restrictive, confer with~\cite{DoukhanLang2009}.

\begin{assumption}[On the joint distributions]\label{ass:smooth}
{ The mappings}
\begin{align*}
&x \mapsto g(x), \quad
x \mapsto \amsmathbb{P}(Z_0 = j \, | \, X = x) g(x), \quad
x \mapsto \amsmathbb{E}[N_{jk}(t) \, | \, X = x] g(x)
\end{align*}
are continously differentiable, with bounded derivatives, in a neighborhood around $x$.
\end{assumption}

{
\begin{assumption}[On the kernel functions]\label{ass:band_kernel_add}
It holds that
\begin{enumerate}[label = (\alph*)]
\item $na_n^{d+4}\to 0$ as $n\to\infty$,
\item For all $i\in\{1,\dots,d\}$ we have $\int K_i(u) \, \mathrm{d}u = 1$ and $\int u K_i(u) \, \mathrm{d}u = 0.$
\end{enumerate}
\end{assumption}
}

\begin{assumption}[Moment condition on the counting processes]\label{ass:bounded}
For $\delta>2$, the mappings
\begin{align*}
x\mapsto \amsmathbb{E}[N_{jk}(t)^{1+\delta} \, | \, X = x] g(x)
\end{align*}
are bounded in neighborhood of $x$.
\end{assumption}
\begin{remark}
The more detailed Assumption~\ref{ass:bounded} is consistent with Assumption~\ref{ass:moment_N} and $\delta>2$.
\end{remark}
Before we state the main result regarding the finite-dimensional distributions, we introduce some additional notation for terms involved in the limiting covariance processes:
\begin{align*}
&p_{j,l}^{\texttt{c}}(s,t | x)
=
\amsmathbb{E}[\mathds{1}_{\{Z_s = j\}}\mathds{1}_{\{s < R\}} \mathds{1}_{\{Z_t = l\}}\mathds{1}_{\{t < R\}} \, | \, X = x], \quad
p_{j}^{\texttt{c}}(s,t | x)
=
p_{j,j}^{\texttt{c}}(s,t | x), \\
&p_{jk}^{\texttt{c}}(s,t | x)
=
\amsmathbb{E}[N_{jk}(s \wedge R)N_{jk}(t \wedge R) \, | \, X = x], \\
&p_{jk,lm}^{\texttt{c}}(s,t | x)
=
\amsmathbb{E}[N_{jk}(s \wedge R)N_{lm}(t \wedge R) \, | \, X = x], \quad
p_{jk}^{2,\texttt{c}}(t | x)
=
p_{jk,jk}^{\texttt{c}}(t,t | x), \\
 &p_{jk}^{2,\texttt{c}}(s,t | x)
 =
\amsmathbb{E}[\{N_{jk}(t \wedge R)-N_{jk}(s \wedge R)\}^2 \, | \, X = x],\\
&p_{jk,l}^{3,\texttt{c}}(s,t | x)
=
\amsmathbb{E}[N_{jk}(s \wedge R)\mathds{1}_{\{Z_t = l\}}\mathds{1}_{\{t < R\}} \, | \, X = x], \quad
p_{jk}^{3,\texttt{c}}(s,t | x)
=
p_{jk,j}^{3,\texttt{c}}(s,t | x).
\end{align*}
In terms of kernel functions, we also define the following shorthand notation.
\begin{align*}
&\mathbb{k}^{(n)}(u,x)
=
\prod_{i=1}^d  \frac{1}{a_n}  K_i\Big(\frac{x_i - u_i}{a_n}\Big), \\
&\phi(x)
=
\frac{\int a_n^{d}\big(\mathbb{k}^{(n)}(u,x)\big)^2 \, \lambda(\mathrm{d}u)}{g(x)}
=
\frac{\int \prod_{i=1}^d  K_i(u_i)^2 \, \lambda(\mathrm{d}u)}{g(x)}.
\end{align*}
Note that $\phi$ is independent of $n$. Finally, we introduce the empirical processes
\begin{align*}
\xi_{jk}^{(n)}(t|x)&=(na_n^{d})^{1/2}\Big(\amsmathbb{N}^{(n)}_{jk}(t|x)-p_{jk}^{\texttt{c}}(t | x)\Big), \\
\xi_{j}^{(n)}(t|x)&=(na_n^{d})^{1/2}\Big(\amsmathbb{I}^{(n)}_{j}(t|x)-p_{j}^{\texttt{c}}(t | x)\Big).
\end{align*}
It is worth noting that the above processes are not centered, though the bias term is seen to vanish in the asymptotic results below. Elements of the proof of the following proposition relate to~\cite{Dabrowska1987}.
\begin{proposition}[Finite dimensional distributions]\label{prop:fidis1}
{It holds that}
\begin{align*}
\big(\xi_{jk}^{(n)}(t|x),\xi_{j}^{(n)}(t|x)\big)_{t\in[0,{\infty)}}\stackrel{\textrm{fd}}{\to}(\xi_{jk}(t|x),\xi_j(t|x))_{t\in[0,{\infty)}},
\end{align*}
the limit being a centered Gaussian process with
{
\begin{align*}
\text{Cov}\big(\xi_{jk}(s|x),\xi_{j}(t|x)\big)&=\phi(x)\big\{p^{3,\texttt{c}}_{jk}(s,t|x)-p^{\texttt{c}}_{jk}(s|x)p^{\texttt{c}}_{j}(t|x)\big\},\\
\text{Cov}\big(\xi_{jk}(s|x),\xi_{jk}(t|x)\big)&=\phi(x)\big\{p^{\texttt{c}}_{jk}(s,t|x)-p^{\texttt{c}}_{jk}(s|x)p^{\texttt{c}}_{jk}(t|x)\big\},\\
\text{Cov}\big(\xi_{j}(s|x),\xi_{j}(t|x)\big)&=\phi(x)\big\{p^{\texttt{c}}_{j}(s,t|x)-p^{\texttt{c}}_{j}(s|x)p^{\texttt{c}}_{j}(t|x)\big\}.
\end{align*}}
\end{proposition}
In similar fashion, we may establish convergence of the finite dimensional distributions across states. It being completely analogous to the previous result, we omit the proof.
\begin{proposition}[Joint finite dimensional distributions]\label{prop:fidis2}
{It holds that}
\begin{align*}
&\big(\xi_{jk}^{(n)}(t|x), \xi_{j}^{(n)}(t|x)\big)_{t\in[0,{\infty)}, j,k \in \mathcal{Z} : j \neq k}
\stackrel{\textrm{fd}}{\to}
(\xi_{jk}(t|x),\xi_{j}(t|x))_{t\in[0,{\infty)}, j,k \in \mathcal{Z} : j \neq k},
\end{align*}
the limit being a centered Gaussian process with
{
\begin{align*}
\text{Cov}(\xi_{jk}(s|x),\xi_{lm}(t|x))&=\phi(x)\big\{p^{\texttt{c}}_{jk,lm}(s,t|x)-p^{\texttt{c}}_{jk}(s|x)p^{\texttt{c}}_{lm}(t|x)\big\}, \\
\text{Cov}(\xi_{jk}(s|x),\xi_{l}(t|x))&=\phi(x)\big\{p^{3,\texttt{c}}_{jk,l}(s,t|x)-p^{\texttt{c}}_{jk}(s|x)p^{\texttt{c}}_{l}(t|x)\big\}, \\
\text{Cov}(\xi_{j}(s|x),\xi_{l}(t|x))&=\phi(x)\big\{p^{\texttt{c}}_{j,l}(s,t|x)-p^{\texttt{c}}_{j}(s|x)p^{\texttt{c}}_{l}(t|x)\big\}.
\end{align*}}
\end{proposition}

\subsection{Weak convergence}\label{subsec:weak}

Having established convergence of the finite-dimensional distributions, we now proceed to use this convergence to prove path-wise weak convergence on a compact interval. 

{Since we are dealing with stochastic process convergence, we refer to~\cite{VaartWellner1996} for a comprehensive introduction. Moreover, since the processes involved are all càdlàg and hence separable, we provide in Appendix~\ref{ap:technical} a modified version of the bracketing central limit theorem appropriate to our situation. Concretely, we establish convergence in
$\amsmathbb{D}[a,b]$, the set of all càdlàg functions $z : [a,b]\to \amsmathbb{R}$, $a,b\in \bar{\amsmathbb{R}}$, endowed with the $J_1$-topology, denoted by $\overset{\mathcal{D}}{\to}$.
} 

\begin{proposition}[Joint asymptotic normality of components]\label{prop:asnorm_comt_joint}
{Suppose $\theta_x$ satisfies~\eqref{eq:right_endpoint}}. Then
\begin{align*}
\big(\xi_{jk}^{(n)}(t|x),\xi_{j}^{(n)}(t|x)\big)_{t\in[0,{\theta_x}], j,k \in \mathcal{Z}: j \neq k}
\stackrel{\mathcal{D}}{\to}
(\xi_{jk}(t|x),\xi_j(t|x))_{t\in[0,{\theta_x}], j,k \in \mathcal{Z}: j \neq k}.
\end{align*}
\end{proposition}
We now turn our attention to the actual empirical processes of interest given by
\begin{align*}
\zeta_{jk}^{(n,\varepsilon)}(t|x) 
=
(na_n^{d})^{1/2}\big(\mathbb{\Lambda}^{(n,\varepsilon)}_{jk}(t|x)-\Lambda^{(\varepsilon)}_{jk}(t | x)\big)\\
\gamma_j^{(n,\varepsilon)}(t|x) 
=
(na_n^{d})^{1/2}\big(\mathbb{p}^{(n,\varepsilon)}_{j}(t|x)-p^{(\varepsilon)}_{j}(t | x)\big).
\end{align*}
The latter relates to the occupation probabilities, while the former concerns the cumulative transition rates. {Below, 
weak limits are obtained for these processes, which in particular shows that the non-parametric rate of $(na_n^{d})^{1/2}$ from simple kernel density estimation is preserved in our setting.}

In the following, let $B^{(\varepsilon)}_j(t | x) = \mathds{1}_{\{p_j(t | x) > \varepsilon\}}$. 
\begin{theorem}[Asymptotic normality of the conditional Nelson--Aalen estimator]\label{thm:asnorm}
{Suppose $\theta_x$ satisfies~\eqref{eq:right_endpoint}}. It holds that
\begin{align*}
\big(\zeta_{jk}^{(n,\varepsilon)}(t|x)\big)_{t\in[0,{\theta_x}]}\stackrel{\mathcal{D}}{\to}\big(\zeta_{jk}^{(\varepsilon)}(t|x)\big)_{t\in[0,{\theta_x}]},
\end{align*}
where
\begin{align}\label{eq:zeta_rep}
\zeta_{jk}^{(\varepsilon)}(t|x)=\int_0^t \frac{1}{p_j^{\texttt{c}}(s-|x) \vee \varepsilon} \,\xi_{jk}(\mathrm{d}s | x)-\int_0^t \frac{B^{(\varepsilon)}_j(s- | x)\xi_j(s-|x)}{p_j^{\texttt{c}}(s-|x)} \, \Lambda^{(\varepsilon)}_{jk}(\mathrm{d}s | x).
\end{align}
\end{theorem}
The formula for the limit covariance from the above theorem is presented at the end of its proof, see Appendix~\ref{appC}.

In a similar fashion, we may establish joint asymptotic normality across states. We omit the proof.
\begin{theorem}[Joint asymptotic normality of estimators]\label{thm:asnorm_joint}
{Suppose $\theta_x$ satisfies~\eqref{eq:right_endpoint}}. It holds that
\begin{align*}
\big(\zeta_{jk}^{(n,\varepsilon)}(t|x)\big)_{t\in[0,{\theta_x}], j,k \in \mathcal{Z}: j \neq k}\stackrel{\mathcal{D}}{\to}\big(\zeta_{jk}^{(\varepsilon)}(t|x)\big)_{t\in[0,{\theta_x}], j,k \in \mathcal{Z}: j \neq k},.
\end{align*}
\end{theorem}
The formula for the limit covariance between states from the above theorem can be calculated as in the proof of Theorem~\ref{thm:asnorm} using the expressions of Theorem~\ref{prop:fidis2}.

{To transfer the asymptotic properties to the product integral, we use the functional delta method. This is possible because the product integral mapping is not only continuous but even Hadamard differentiable, confer with~\cite{GillJohansen1990}.}
\begin{theorem}[Asymptotic normality of the conditional Aalen--Johansen estimator]\label{thm:asnorm2}
{Suppose $\theta_x$ satisfies~\eqref{eq:right_endpoint}}. It holds that
\begin{align*}
\big(\gamma^{(n,\varepsilon)}(t|x)\big)_{t\in[0,{\theta_x}]}\stackrel{\mathcal{D}}{\to}\big(\gamma^{(\varepsilon)}(t|x)\big)_{t\in[0,{\theta_x}]},
\end{align*}
where
\begin{align*}
\gamma^{(\varepsilon)}(t|x)
=
p^{(\varepsilon)}(0 | x)\int_0^t \Prodi_0^s \big(\text{Id} +  \Lambda^{(\varepsilon)}(\mathrm{d}u | x) \big) 
\zeta^{(\varepsilon)}(\mathrm{d}s | x)
\Prodi_s^t \big(\text{Id} +  \Lambda^{(\varepsilon)}(\mathrm{d}u | x) \big).
\end{align*}
\end{theorem}
The covariance processes associated to the Gaussian processes in Theorems~\ref{thm:asnorm} and~\ref{thm:asnorm2} are computationally {challenging}. In principle, one may construct a Monte-Carlo covariance estimator based on simulations of $\xi_j$ and $\xi_{jk}$, using kernel-based plug-in estimators for the unknown quantities. {However, this is computationally expensive, so in Subsection~\ref{subsec:plugin} we present an alternative solution.}

\subsection{Plug-in estimators for the covariances}\label{subsec:plugin}
Introduce $\tilde{\zeta}_{jk}^{(\varepsilon)}(\cdot|x)$ according to
\begin{align*}
\tilde{\zeta}_{jk}^{(\varepsilon)}(t|x)
&=
\sqrt{\phi(x)} \bigg(
\int_0^t \frac{1}{p_j^{\texttt{c}}(s-|x) \vee \varepsilon} \big(N_{jk}(\mathrm{d}s \wedge R)  - p_{jk}^{\texttt{c}}(\mathrm{d}s | x)\big) \\ 
&\quad - \int_0^t \frac{B^{(\varepsilon)}_j(s- | x)\big(\mathds{1}_{\{s \leq R\}}\mathds{1}_{\{Z_{s-}=j\}}-p_j^\texttt{c}(s- | x)\big)}{p_j^{\texttt{c}}(s-|x)} \, \Lambda^{(\varepsilon)}_{jk}(\mathrm{d}s | x) \bigg),
\end{align*}
which is similar to~\eqref{eq:zeta_rep}, but with the empirical processes replaced by single random variables. It is possible to show that
\begin{align*}
\Sigma_{\Lambda_{jk}^{(\varepsilon)}}(s,t | x)
:=
\text{Cov}\big(\zeta_{jk}^{(\varepsilon)}(s|x),\zeta_{jk}^{(\varepsilon)}(t|x)\big)
=
\amsmathbb{E}\big[\tilde{\zeta}_{jk}^{(\varepsilon)}(s|x)\tilde{\zeta}_{jk}^{(\varepsilon)}(t|x) \, \big| \, X = x\big],
\end{align*}
for instance by calculating the right-hand side explicitly and referring to the covariances in the proof of Theorem~\ref{thm:asnorm}. This immediately gives rise to a natural and strongly consistent kernel-based plug-in estimator $\mathbb{\Sigma}_{\Lambda_{jk}^{(\varepsilon)}}^{(n)}(\cdot | x)$ of $\Sigma_{\Lambda_{jk}^{(\varepsilon)}}(\cdot | x)$ given by
\begin{align*}
\mathbb{\Sigma}_{\Lambda_{jk}^{(\varepsilon)}}^{(n)}(s, t | x)
=
\frac{1}{n} \sum_{\ell = 1}^n \frac{\tilde{\zeta}_{jk}^{(n,\ell,\varepsilon)}(s|x)\tilde{\zeta}_{jk}^{(n,\ell,\varepsilon)}(t|x)  \mathbb{g}^{(n,\ell)}(x)}{\mathbb{g}^{(n)}(x)},
\end{align*}
where $\tilde{\zeta}_{jk}^{(n,\ell,\varepsilon)}(\cdot|x)$ is equal to $\tilde{\zeta}_{jk}^{(\varepsilon)}(\cdot|x)$, but with $N_{jk}$, $Z$, and $R$ replaced by $N_{jk}^\ell$, $Z^\ell$, and $R^\ell$ as well as all theoretical quantities replaced by their kernel estimators:
\begin{align*}
\tilde{\zeta}_{jk}^{(n,\ell,\varepsilon)}(t|x)
&=
\sqrt{\tilde{\phi}^{(n)}(x)} \bigg(
\int_0^t \frac{1}{\amsmathbb{I}_j^{(n)}(s- | x) \vee \varepsilon} \big(N^\ell_{jk}(\mathrm{d}s \wedge R^\ell)  - \amsmathbb{N}_{jk}^{(n)}(\mathrm{d}s | x)\big) \\ 
&\quad - \int_0^t \frac{\mathds{1}_{\{\amsmathbb{I}_j^{(n)}(s- | x)>\varepsilon\}}\big(\mathds{1}_{\{s \leq R^\ell\}}\mathds{1}_{\{Z^\ell_{s-}=j\}}-\amsmathbb{I}_j^{(n)}(s- | x)\big)}{\amsmathbb{I}_j^{(n)}(s- | x)} \, \mathbb{\Lambda}_{jk}^{(n,\varepsilon)}(\mathrm{d}s | x) \bigg),
\end{align*}
where $\tilde{\phi}^{(n)}$ is the following kernel-based plug-in estimator of $\phi$:
\begin{align*}
\tilde{\phi}^{(n)}(x)
=
\frac{\int \prod_{i=1}^d  K_i(u_i)^2 \, \lambda(\mathrm{d}u)}{\mathbb{g}^{(n)}(x)}.
\end{align*}
Similarly, 
\begin{align*}
\Sigma_{p_j^{(\varepsilon)}}(s,t | x)
:=
\text{Cov}\big(\gamma_j^{(\varepsilon)}(s|x),\gamma_j^{(\varepsilon)}(t|x)\big)
=
\amsmathbb{E}\big[\tilde{\gamma}_j^{(\varepsilon)}(s|x)\tilde{\gamma}_j^{(\varepsilon)}(t|x) \, \big| \, X = x\big],
\end{align*}
where $\tilde{\gamma}_j^{(\varepsilon)}(\cdot|x)$ is equal to $\gamma_j^{(\varepsilon)}(\cdot|x)$, but with $\zeta^{(\varepsilon)}(\cdot | x)$ replaced by $\tilde{\zeta}^{(\varepsilon)}(\cdot | x)$. Again, this immediately gives rise to a natural and strongly consistent kernel-based plug-in estimator $\mathbb{\Sigma}_{p_j^{(\varepsilon)}}^{(n)}(\cdot | x)$ of $\Sigma_{p_j^{(\varepsilon)}}(\cdot | x)$ given by
\begin{align*}
\mathbb{\Sigma}_{p_j^{(\varepsilon)}}^{(n)}(s, t | x)
=
\frac{1}{n} \sum_{\ell = 1}^n \frac{\tilde{\gamma}_j^{(n,\ell,\varepsilon)}(s|x)\tilde{\gamma}_j^{(n,\ell,\varepsilon)}(t|x)  \mathbb{g}^{(n,\ell)}(x)}{\mathbb{g}^{(n)}(x)},
\end{align*}
where $\tilde{\gamma}^{(n,\ell,\varepsilon)}(\cdot|x)$ is equal to $\tilde{\gamma}^{(\varepsilon)}(\cdot|x)$, but with all theoretical quantities replaced by their estimators:
\begin{align*}
\tilde{\gamma}^{(n,\ell,\varepsilon)}(t|x)
=
\amsmathbb{I}^{(n)}(0 | x)\int_0^t \Prodi_0^s \big(\text{Id} +  \mathbb{\Lambda}^{(n,\varepsilon)}(\mathrm{d}u | x) \big) 
\tilde{\zeta}_{jk}^{(n,\ell,\varepsilon)}(\mathrm{d}s|x)
\Prodi_s^t \big(\text{Id} +  \mathbb{\Lambda}^{(n,\varepsilon)}(\mathrm{d}u | x) \big).
\end{align*}
In case of no covariates, this estimator closely resembles one of the estimators proposed in~\cite{Glidden2002}, {where consistency is established. An extension to our setting is a straightforward adaptation of the unconditional to the conditional case following our proofs techniques, and thus we omit the details.}

%

\begin{appendix}
\section{Proofs of consistency}\label{appA}

\begin{proof}[Proof of Proposition~\ref{prop:N_I_const}]
We first note that arguments analogous to the classic proof of the Glivenko--Cantelli theorem extend~\eqref{eq:censoring_point} to uniform convergence. Referring to the decomposition~\eqref{eq:indicator_decomp}, and invoking also~\eqref{eq:indicator_point}, it thus suffices to establish~\eqref{eq:count_uniform}. To this end, let $C_{m,n}$ and $C_m$ be given by
\begin{align*}
C_{m,n} = \bigg\{ \sum_{j'{,}k'} \amsmathbb{N}^{(n)}_{j'k'}({\theta_x} | x) \leq m\bigg\}, \quad C_m = \Big\{ C_{m,n} \text{ evt.} \Big\}.
\end{align*}
Recall that $\amsmathbb{N}^{(n)}_{jk}(\cdot | x)$ is increasing. Since $\amsmathbb{N}^{(n)}_{jk}(\cdot | x)$ is eventually bounded on $C_m$, arguments analogous to the classic proof of the Glivenko--Cantelli theorem again extend~\eqref{eq:count_point} to~\eqref{eq:count_uniform} on $C_m$. Referring to~\eqref{eq:count_point}, we have for $m$ sufficiently large that $\amsmathbb{P}(C_m) = 1$. This completes the proof.
\end{proof}

\begin{proof}[Proof of Theorem \ref{theo:L_P_const}]
Note that
\begin{align*}
&\sup_{0 \leq t \leq {\theta_x}} \big| \mathbb{\Lambda}_{jk}^{(n,\varepsilon)}(t | x) - \Lambda_{jk}^{(\varepsilon)}(t | x) \big| \\
&\leq
\sup_{0 < t \leq {\theta_x}}
\bigg| \int_0^t \bigg( \frac{1}{\amsmathbb{I}^{(n)}_j(s-|x) \vee \varepsilon} - \frac{1}{p_j^{\texttt{c}}(s-|x) \vee \varepsilon}\bigg) \, \amsmathbb{N}_{jk}^{(n)}(\mathrm{d}s | x) \bigg| \\
&\quad +
\sup_{0 \leq t \leq {\theta_x}}
\bigg| \int_0^t \frac{1}{p_j^{\texttt{c}}(s-|x) \vee \varepsilon} \Big( \amsmathbb{N}_{jk}^{(n)}(\mathrm{d}s | x) - p_{jk}^{\texttt{c}}(\mathrm{d}s|x)\Big) \bigg|.
\end{align*}
Regarding the first term, it holds that
\begin{align*}
&\sup_{0 \leq t \leq {\theta_x}}
\bigg| \int_0^t \bigg( \frac{1}{\amsmathbb{I}^{(n)}_j(s-|x) \vee \varepsilon} - \frac{1}{p_j^{\texttt{c}}(s-|x) \vee \varepsilon}\bigg) \, \amsmathbb{N}_{jk}^{(n)}(\mathrm{d}s | x) \bigg| \\
&\leq
\amsmathbb{N}_{jk}^{(n)}({\theta_x} | x) \sup_{0  t \leq {\theta_x}} \bigg| \frac{1}{\amsmathbb{I}^{(n)}_j({t}|x) \vee \varepsilon} - \frac{1}{p_j^{\texttt{c}}({t}|x) \vee \varepsilon} \bigg|.
\end{align*}
It follows from~\eqref{eq:count_point}, in conjunction with~\eqref{eq:indicator_uniform} and the fact that $[0,1] \ni x \mapsto \frac{1}{x\vee\varepsilon}$ is uniformly continuous, that
\begin{align*}
\amsmathbb{N}_{jk}^{(n)}({\theta_x} | x) \sup_{0 \leq t \leq {\theta_x}} \bigg| \frac{1}{\amsmathbb{I}^{(n)}_j({t}|x) \vee \varepsilon} - \frac{1}{p_j^{\texttt{c}}({t}|x) \vee \varepsilon} \bigg| \overset{\text{a.s.}}{\to} 0, \quad n \to \infty.
\end{align*}
Turning our attention to the second term, note that $\frac{1}{p_j^{\texttt{c}}(\cdot | x) \vee \varepsilon}$ is right-continuous and of finite variation, since this is the case for $p_j^{\texttt{c}}( \cdot | x)$, confer with~\eqref{eq:pc_decomp}. We may therefore apply integration by parts to show that
\begin{align*}
&{\bigg|}\int_0^t \frac{1}{p_j^{\texttt{c}}(s-|x) \vee \varepsilon} \Big( \amsmathbb{N}_{jk}^{(n)}(\mathrm{d}s | x) - p_{jk}^{\texttt{c}}(\mathrm{d}s|x)\Big){\bigg|} \\
&=
{\bigg|}\frac{1}{p_j^{\texttt{c}}(t|x) \vee \varepsilon} \Big( \amsmathbb{N}_{jk}^{(n)}(t | x) - p_{jk}^{\texttt{c}}(t | x) \Big) \\
&\quad- \int_0^t \Big( \amsmathbb{N}_{jk}^{(n)}(s | x) - p_{jk}^{\texttt{c}}(s | x) \Big)  \, \mathrm{d}\bigg(\frac{1}{p_j^{\texttt{c}}(s | x) \vee \varepsilon}\bigg) {\bigg|} \\
&\leq \frac{1}{\varepsilon} \Big| \amsmathbb{N}_{jk}^{(n)}(t | x) - p_{jk}^{\texttt{c}}(t | x) \Big| + \bigg\| \frac{1}{p_j^{\texttt{c}}(\cdot | x) \vee \varepsilon} \bigg\|_{v_1} \sup_{0 \leq s \leq t} \Big| \amsmathbb{N}_{jk}^{(n)}(s | x) - p_{jk}^{\texttt{c}}(s | x) \Big|
\end{align*}
with $\| f \|_{v_1}$ denoting the variation of $f$ on $[0,{\theta_x}]$. In conclusion,
\begin{align*}
&\sup_{0 < t \leq {\theta_x}}
\bigg| \int_0^t \frac{1}{p_j^{\texttt{c}}(s-|x) \vee \varepsilon} \Big( \amsmathbb{N}_{jk}^{(n)}(\mathrm{d}s | x) - p_{jk}^{\texttt{c}}(\mathrm{d}s|x)\Big) \bigg| \\
&\leq
\bigg(\frac{1}{\varepsilon} + \bigg\| \frac{1}{p_j^{\texttt{c}}(\cdot | x) \vee \varepsilon} \bigg\|_{v_1} \bigg) \sup_{0 \leq t \leq {\theta_x}} \Big| \amsmathbb{N}_{jk}^{(n)}(t | x) - p_{jk}^{\texttt{c}}(t | x) \Big| \to 0, \quad n \to \infty,
\end{align*}
confer with~\eqref{eq:count_uniform}. This proves the first part of the theorem. 

Next, note that by~\eqref{eq:count_point},
\begin{align*}
\limsup\limits_{n\to{\infty}} \big\| \mathbb{\Lambda}^{(n,\varepsilon)}(\cdot | x) \big\|_{v_1}
=
\limsup\limits_{n\to{\infty}} \max_{j\in\mathcal{Z}} 2 \sum_{k \in \mathcal{Z}\atop k \neq j} \mathbb{\Lambda}_{jk}^{(n,\varepsilon)}({\theta_x} | x)
=
2 \max_{j\in\mathcal{Z}} \sum_{k \in \mathcal{Z}\atop k \neq j} \Lambda_{jk}^{(\varepsilon)}({\theta_x} \, | \, x) < \infty.
\end{align*}
From Theorem~7 in~\cite{GillJohansen1990}, we may then extend the uniform convergence of $\mathbb{\Lambda}^{(n,\varepsilon)}(\cdot | x)$ to its product integral, that is, in particular,
\begin{align*}
\sup_{0 < t \leq {\theta_x}} \bigg\Vert \Prodi_0^t \big(\text{Id} + \mathbb{\Lambda}^{(n,\varepsilon)}(\mathrm{d}s | x)\big) -  \Prodi_0^t \big(\text{Id} + \Lambda^{(\varepsilon)}(\mathrm{d}s | x)\big) \bigg\Vert  \overset{\text{a.s.}}{\to} 0, \quad n \to \infty.
\end{align*}
Together with the convergence of $\amsmathbb{I}_j^{(n)}(0 | x)$ to $p_j(0 | x)$ according to~\eqref{eq:indicator_point}, this establishes the second part of the theorem and thus concludes the proof.
\end{proof}

\section{Consistency of density estimator}\label{appB}
\begin{proof}[Proof of Proposition \ref{prop:kernel_atoms}]
We proceed by reduction to the case covered in~\cite{Stute1986}. To this end, assume without loss of generality that we may decompose
\begin{align*}
x=(x_1,\dots,x_{d_1},x_{d_1+1},\dots,x_d)=(x^a,x^c),
\end{align*}
where the first $d_1$ coordinates are potential known atoms, say $\alpha_1,\ldots,\alpha_{d_1}$, while the remaining coordinates are not. Introduce the conditional notation $g(x)=g^{a|c}(x^a|x^c) g^c(x^c)$. Straightforward calculations yield
\begin{align*}
\mathbb{g}^{(n)}(x)&=
\frac{1}{n}
\sum_{\ell = 1}^n
\prod_{i=1}^{d_1} \mathds{1}_{\{X_i^\ell = \alpha_i\}} \prod_{i=d_1+1}^d \frac{1}{a_n}  K_i\Big(\frac{x_i - X_i^\ell}{a_n}\Big) \mathds{1}_{\{X_i^\ell \notin \alpha \}}\\
&=
\frac{1}{n}
\sum_{\ell = 1}^n
 \mathds{1}_{\{X_i^\ell = \alpha_i,\,i=1,\dots,d_1\}}  \mathds{1}_{\{X_i^\ell \notin \alpha,\, i=d_1+1,\dots, d \}}\frac{1}{a_n^{d-d_1}} \tilde{K}\Big(\frac{x^c - X^{\ell,c}}{a_n}\Big),
\end{align*}
where
\begin{align*}
\frac{1}{a_n^{d-d_1}}\tilde{K}\Big( \frac{x^c-X^{\ell,c}}{a_n}\Big):= \prod_{i=d_1+1}^d\frac{1}{a_n}  K_i\Big(\frac{x_i - X_i^\ell}{a_n}\Big)
\end{align*}
conforms to the setting of~\cite{Stute1986}. It follows that
\begin{align*}
\mathbb{g}^{(n)}(x)&=
\frac{
\frac{1}{n}
\sum_{\ell = 1}^n
 \mathds{1}_{\{X_i^\ell = \alpha_i,\,i=1,\dots,d_1\}}  \mathds{1}_{\{X_i^\ell \notin \alpha,\, i=d_1+1,\dots, d \}}\frac{1}{a_n^{d-d_1}}\tilde{K}\left( \frac{x^c-X^{\ell,c}}{a_n}\right)
 }
 {
 \frac{1}{n}\sum_{\ell=1}^n\frac{1}{a_n^{d-d_1}}\tilde{K}\left( \frac{x^c-X^{\ell,c}}{a_n}\right)
 }\\
 &\quad \times \frac{1}{n}\sum_{\ell=1}^n\frac{1}{a_n^{d-d_1}}\tilde{K}\left( \frac{x^c-X^{\ell,c}}{a_n}\right)\\
&\overset{\text{a.s.}}{\to} g^{a|c}(x^a|x^c) g^c(x^c)=g(x), \quad n \to \infty,
\end{align*}
by an application of Theorem~3 in~\cite{Stute1986} to $\mathds{1}_{\{X_i^\ell = \alpha_i,\,i=1,\dots,d_1\}}  \mathds{1}_{\{X_i^\ell \notin \alpha,\, i=d_1+1,\dots, d \}}$. In similar fashion, using the previous result, the moment condition on $1$, and applying the aforementioned theorem again, we obtain that
\begin{align*}
&\frac{1}{n}\frac{\sum_{\ell=1}^n Y^\ell\: \mathbb{g}^{(n,\ell)}(x)}{\mathbb{g}^{(n)}(x)} \\
&=\frac{
\frac{1}{n}
\sum_{\ell = 1}^n Y^\ell\,
 \mathds{1}_{\{X_i^\ell = \alpha_i,\,i=1,\dots,d_1\}}  \mathds{1}_{\{X_i^\ell \notin \alpha,\, i=d_1+1,\dots, d \}}\frac{1}{a_n^{d-d_1}}\tilde{K}\Big( \frac{x^c-X^{\ell,c}}{a_n}\Big)
 }{\frac{1}{n}\sum_{\ell=1}^n\frac{1}{a_n^{d-d_1}}\tilde{K}\Big( \frac{x^c-X^{\ell,c}}{a_n}\Big)
}\\
 &\quad\times 
 \frac{\frac{1}{n}\sum_{\ell=1}^n\frac{1}{a_n^{d-d_1}}\tilde{K}\Big( \frac{x^c-X^{\ell,c}}{a_n}\Big)
}{\frac{1}{n}
\sum_{\ell = 1}^n
 \mathds{1}_{\{X_i^\ell = \alpha_i,\,i=1,\dots,d_1\}}  \mathds{1}_{\{X_i^\ell \notin \alpha,\, i=d_1+1,\dots, d \}}\frac{1}{a_n^{d-d_1}}\tilde{K}\Big( \frac{x^c-X^{\ell,c}}{a_n}\Big)}\\
&\overset{\text{a.s.}}{\to} \frac{\amsmathbb{E}[Y\mathds{1}_{\{X_i^\ell = \alpha_i,\,i=1,\dots,d_1\}} \, | \, X^c=x^c]g^c(x^c)}{g(x)}=\amsmathbb{E}[Y \, | \, X = x], \quad n \to \infty,
\end{align*}
as desired.
\end{proof}

\section{Proofs of normality}\label{appC}
\begin{proof}[Proof of Proposition~\ref{prop:fidis1}]
The proof relies on a bias-variance decomposition. Set
\begin{align*}
g^{(n)}(x)&=\int g(u) \mathbb{k}^{(n)}(u,x) \, \lambda(\mathrm{d}u), \\
p^{(\texttt{c},n)}_{j}(t|x)&=\frac{\int p_j^{\texttt{c}}(t|u)g(u) \mathbb{k}^{(n)}(u,x) \, \lambda(\mathrm{d}u)}{g^{(n)}(x)}, \\
p^{(\texttt{c},n)}_{jk}(t|x)&=\frac{\int p_{jk}^{\texttt{c}}(t|u)g(u) \mathbb{k}^{(n)}(u,x) \, \lambda(\mathrm{d}u)}{g^{(n)}(x)}.
\end{align*}
First, notice that $\amsmathbb{E}[\mathbb{g}^{(n)}(x)]=g^{(n)}(x)$, while, uniformly in $t$, $\amsmathbb{E}[\amsmathbb{N}^{(n)}_{jk}(t|x)]=p^{(\texttt{c},n)}_{jk}(t|x)+O((na_n^{d})^{-1})$ and $\amsmathbb{E}[\amsmathbb{I}^{(n)}_{jk}(t|x)]=p^{(\texttt{c},n)}_{j}(t|x)+O((na_n^{d})^{-1})$.

By Assumption~\ref{ass:smooth}, we may write the multivariate Taylor expansion under the integral sign, so with $p_{jk}^{\texttt{c}}(t,u):= p_{jk}^{\texttt{c}}(t|u)g(u)$, we get
\begin{align*}
p^{(\texttt{c},n)}_{jk}(t|x)
=\frac{ p_{jk}^{\texttt{c}}(t,x)}{g^{(n)}(x)}+\frac{a_n^2}{2} \sum_{m=1}^{d}\frac{\partial^2 p_{jk}^{\texttt{c}}(t,x)}{\partial x_m^2}\,\frac{\int u^2_m \mathbb{k}^{(n)}(u,x)g(u) \, \lambda(\mathrm{d}u)}{g^{(n)}(x)}+o(a_n^2),
\end{align*}
where the linear term has vanished due Assumption~\ref{ass:band_kernel_add}(b). Thus, by a simple Taylor expansion of $g^{(n)}(x)$, we further get
\begin{align*}
p^{(\texttt{c},n)}_{jk}(t|x)-p^{\texttt{c}}_{jk}(t|x)=O(a_n^2)
\end{align*}
uniformly in $t$, and an analogous expressions holds for the difference $p^{(\texttt{c},n)}_{j}(t|x)-p^{\texttt{c}}_{j}(t|x)$.

It follows by Slutsky's theorem that is suffices to establish the normality of the finite dimensional distributions of the conditional empirical processes 
\begin{align*}
\tilde\xi_{jk}^{(n)}(t|x)&=(na_n^{d})^{1/2}\Big(\amsmathbb{N}^{(n)}_{jk}(t|x)-p_{jk}^{(\texttt{c},n)}(t | x)\Big),\\
\tilde\xi_{j}^{(n)}(t|x)&=(na_n^{d})^{1/2 }\Big(\amsmathbb{I}^{(n)}_{j}(t|x)-p_{j}^{(\texttt{c},n)}(t | x)\Big),
\end{align*}
since the remaining bias terms go to zero whenever $(na_n^{d})^{1/2} a_n^2=(na_n^{d+4})^{1/2}\to0$, which is equivalent to Assumption~\ref{ass:band_kernel_add}(a).

We proceed to show convergence of the variance term. This is best handled in a joint rather than a conditional manner. To this end, we define $p^{(\texttt{c},n)}_{jk}(t,x)=p^{(\texttt{c},n)}_{jk}(t|x)g^{(n)}(x)$ and $\amsmathbb{N}^{(n)}_{jk}(t,x)=\amsmathbb{N}^{(n)}_{jk}(t|x)\mathbb{g}^{(n)}(x)$ and similarly for $\tilde\xi_{jk}^{(n)}(t,x)$ and $\tilde\xi_{j}^{(n)}(t,x)$. We also introduce
\begin{align*}
\tilde\xi^{(n)}(x)=(na_n^{d})^{1/2}\big(\mathbb{g}^{(n)}(x)-g^{(n)}(x)\big).
\end{align*}
Let $t_1<\cdots<t_M$ and $c_m,d_m$, $m=1,\ldots,M$ satisfy $\sum_{m=1}^M (c_m^2+d_m^2)>0.$ Then the quantities $n^{1/2}\tilde\xi^{(n)}(x)$ and
\begin{align*}
&n^{1/2}\sum_m \left(c_m\tilde\xi_{jk}^{(n)}(t_m,x) + d_m \tilde\xi_{j}^{(n)}(t_m,x)\right)\\
&=a_n^{d/2}\sum_m \sum_{\ell=1}^n \Big(c_m\left\{N_{jk}^\ell(t_m\wedge R^\ell)\mathbb{g}^{(n,\ell)}(x) - p^{(\texttt{c},n)}_{jk}(t_m,x) \right\}\\
&\quad\quad\quad\quad\quad\quad\quad +d_m \left\{\mathds{1}_{\{t_m < R^\ell\}} \mathds{1}_{\{Z^\ell_{t_m} = j\}} \mathbb{g}^{(n,\ell)}(x)-p^{(\texttt{c},n)}_{j}(t_m,x)\right\} \Big),
\end{align*}
are sums of i.i.d.\ terms with means exactly zero. Notice that each of the terms in the parenthesis in the last sum is of order $O(a_n^{-d/2})$. To check that the limit of the sum is non-degenerate, we examine the covariance functions, term by term. In what follows, we apply a Taylor expansion and use the convergence of deterministic terms as per Theorem~2 in~\cite[pp.\ 62-63]{Stein1970}.

We first obtain
\begin{align*}
&\text{Cov}\bigg(\sum_m \Big(c_m\tilde\xi_{jk}^{(n)}(t_m,x) + d_m \tilde\xi_{j}^{(n)}(t_m,x)\Big), \tilde\xi^{(n)}(x) \bigg)\\
&=\sum_m \bigg(c_m \text{Cov}\Big(\tilde\xi_{jk}^{(n)}(t_m,x) ,\tilde\xi^{(n)}(x)  \Big) + d_m \text{Cov}\Big(\tilde\xi_{j}^{(n)}(t_m,x) ,\tilde\xi^{(n)}(x) \Big)\bigg)\\
&=a_n^{d} \sum_m c_m 
\int p_{jk}^{\texttt{c}}(t_{m} | u)g(u) \mathbb{k}^{(n)}(u,x)^2 \lambda(\mathrm{d}u)\\
&\quad-a_n^{d} \sum_m c_m 
\bigg(\int p_{jk}^{\texttt{c}}(t_{m} | u)g(u) \mathbb{k}^{(n)}(u,x) \, \lambda(\mathrm{d}u)\bigg)\bigg(\int g(u) \mathbb{k}^{(n)}(u,x) \, \lambda(\mathrm{d}u)\bigg) \\
&\quad+a_n^{d} \sum_m d_m 
\int p_{j}^{\texttt{c}}(t_{m} | u)g(u) \mathbb{k}^{(n)}(u,x)^2 \lambda(\mathrm{d}u)\\
&\quad-a_n^{d} \sum_m d_m 
\bigg(\int p_{j}^{\texttt{c}}(t_{m} | u)g(u) \mathbb{k}^{(n)}(u,x) \, \lambda(\mathrm{d}u)\bigg)\bigg(\int g(u) \mathbb{k}^{(n)}(u,x)\, \lambda(\mathrm{d}u)\bigg),
\end{align*}
which is of order 
\begin{align*}
&\sum_{m} \big(c_{m}\phi(x)p_{jk}^{\texttt{c}}(t_{m} | x)g(x)^2+d_{m} \phi(x)p_{j}^{\texttt{c}}(t_{m} | x)g(x)^2\big)+o(1).
\end{align*}
Similarly,
\begin{align*}
&\text{Var}\bigg(\sum_m \Big(c_m\tilde\xi_{jk}^{(n)}(t_m,x) + d_m \tilde\xi_{j}^{(n)}(t_m,x)\Big)\bigg)\\
&=\sum_{m_1,m_2} \big(c_{m_1}c_{m_2}\phi(x)p_{jk}^{\texttt{c}}(t_{m_1},t_{m_2} | x)g(x)^2+d_{m_1}d_{m_2} \phi(x)p_{j}^{\texttt{c}}(t_{m_1},t_{m_2} | x)g(x)^2\big) \\
&\quad+\sum_{m_1,m_2} (c_{m_1}d_{m_2} \phi(x) p_{jk}^{3,\texttt{c}}(t_{m_1},t_{m_2} | x)g^2(x)+c_{m_2}d_{m_1} \phi(x) p_{jk}^{3,\texttt{c}}(t_{m_2},t_{m_1} | x)g^2(x)) + o(1)
\end{align*}
as well as
\begin{align*}
\text{Var}\Big(\tilde\xi^{(n)}(x) \Big)
=
a_n^{d} \sum_m  \bigg(
\int \mathbb{k}^{(n)}(u,x)^2 \lambda(\mathrm{d}u)-
\Big(\int \mathbb{k}^{(n)}(u,x) \, \lambda(\mathrm{d}u)\Big)^{\!2}\bigg)
=g(x)^2\phi(x)+o(1).
\end{align*}
This confirms non-degeneracy. It follows by the {Berry--}Esseen theorem, using Assumption~\ref{ass:bounded} that the above linear combinations are asymptotically normal. But then, by the Cramer--Wold device, so are the finite-dimensional distributions of the joint empirical processes. {Upon division, and using Slutsky's theorem, we obtain the required covariance expressions for the conditional processes.}

The actual calculation of the limiting covariance processes is straightforward, following the above expansions.
\end{proof}

\begin{proof}[Proof of Proposition~\ref{prop:asnorm_comt_joint}]

It suffices, due to~\eqref{eq:decomp} and~\eqref{eq:indicator_decomp}, to show weak convergence of $\big(\xi_{jk}^{(n)}(t|x)\big)_{t\in[0,\theta_x], j,k \in \mathcal{Z} : j \neq k}$ together with the empirical processes associated with the estimators $\amsmathbb{C}_j^{(n)}(\cdot | x)$, $j\in \mathcal{Z}$. Convergence of the joint finite-dimensional distributions follows similarly to Proposition~\ref{prop:fidis2}, so only asymptotic tightness (or asymptotic equicontinuity) of each individual process is required. To this end, we apply Theorem~\ref{clt_changingfunc} for $\xi_{jk}^{(n)}(\cdot|x)$, the other case following analogously.

Consider the i.i.d.\ sequence $(\mathcal{X}^\ell)_{\ell = 1}^n$ of separable processes given by
\begin{align*}
\mathcal{X}^\ell=(X^\ell,(Z^\ell_t)_{0\le t\le R^\ell}, \tau^\ell \wedge R^\ell).
\end{align*}
We apply Theorem~\ref{clt_changingfunc} with $T=[0,\theta_x]$, 
\begin{align*}
f_{n,t}(\mathcal{X}^\ell)=\frac{1}{a_n^{d/2}}\prod_{i=1}^d K_i\Big(\frac{x_i-X_i^\ell}{a_n}\Big) N_{jk}^\ell(t \wedge R^\ell),
\end{align*}
and envelope function
\begin{align*}
{F_n(\mathcal{X})}=\frac{1}{a_n^{d/2}}\prod_{i=1}^d K_i\Big(\frac{x_i-X_i}{a_n}\Big) N_{jk}({\theta_x} \wedge R).
\end{align*}
For this, we need to verify conditions~\eqref{changingclt_12} and~\eqref{changingclt_4} for a partition chosen independent of $n$.

We shall require to an analysis of the behavior of the $\mu ||{F_n(\mathcal{X})}||_{2}$-bracketing number, but this does not cause complications, since
\begin{align}\label{eq:help}
\mu^2 ||{F_n(\mathcal{X})}||^2_{2}=\mu^2 \phi(x) g^2(x) p^{2,c}_{jk}(\theta|x)+o(1),
\end{align}
where the term accompanying $\mu^2$ is constant in $n$ and strictly positive (whenever relevant). Consequently, we may just use our analysis for the $\mu$-case.

From~\eqref{eq:help}, it holds that $||{F_n(\mathcal{X})}||^2_{2} = O(1)$. Let $\eta>0$. We may calculate
\begin{align*}
&\amsmathbb{E}[{F_n(\mathcal{X})}^2 \mathds{1}_{\{{F_n(\mathcal{X})} > \eta \sqrt{n}\}}]\\
&=
a_n^d \int_{\amsmathbb{R}^d}
\mathbb{k}^{(n)}(u,x)^2 \amsmathbb{E}\big[N_{jk}({\theta_x} \wedge R)^2 \mathds{1}_{\{K((x-u)a_n^{-1}) N_{jk}({\theta_x}) > \eta \sqrt{a_n^d n}\}} \, \big| \, X = u\big] g(u) \, \lambda(\mathrm{d}u).
\end{align*}
Since the kernel is assumed to have bounded variation and support, it is in particular bounded. Recall that $N_{jk}({\theta_x})$ is square integrable. Then, by dominated convergence and the fact that $a_n^d n \to \infty$, 
\begin{align*}
\lim_{n\to\infty}\amsmathbb{E}\big[N_{jk}({\theta_x} \wedge R)^2 \mathds{1}_{\{K((x-u)a_n^{-1}) N_{jk}({\theta_x}) > \eta \sqrt{a_n^d n}\}} \, \big| \, X = u\big]
= 
0.
\end{align*}
This verifies~\eqref{changingclt_12}.

Finally, we calculate the bracketing number. Let $0 = t_0 < t_1 < \cdots < t_M = {\theta_x}$ and note that
\begin{align*}
\sum_{\ell=1}^n \amsmathbb{E}\bigg[ \sup_{t_{\ell-1} < s \leq t \leq t_\ell}\Big|f_{n,t}(\mathcal{X}^\ell)-f_{n,s}(\mathcal{X}^\ell) \Big|^2 \bigg]
=
\phi(x)g^2(x) \sum_{\ell=1}^n p_{jk}^{2,\texttt{c}}(t_\ell,t_{\ell-1} | x) + o(1).
\end{align*}
Following closely the chain of arguments in Example~2.11.16 of~\cite{VaartWellner1996}, there exists a constant $C>0$ independent of $n$ such that, with less than $C/\mu^2$ terms in the partition, it holds that
\begin{align*}
\phi(x)g^2(x) \sum_{\ell=1}^n p_{jk}^{2,\texttt{c}}(t_\ell,t_{\ell-1} | x) \leq \mu^2/2.
\end{align*}
This in particular implies that, for $n$ sufficiently large, the bracketing number is less than $C/\mu^2$. Consequently, for $n$ large enough,
\begin{align*}
\int_0^{\delta_n}\sqrt{\log{N_{[\:]}(\mu,\mathcal{F}_n,L_2(\amsmathbb{P}(\mathcal{X}^1)))}} \, \mathrm{d}\mu
\leq
\int_0^{\delta_n} \sqrt{\log{C/\mu^2}} \, \mathrm{d}\mu.
\end{align*}
Using that $\sqrt{y}/\log{Cy} \to \infty$ for $y \to \infty$, for $n$ large enough it holds that
\begin{align*}
\int_0^{\delta_n} \sqrt{\log{C/\mu^2}} \, \mathrm{d}\mu
\leq
\int_0^{\delta_n} \frac{1}{\sqrt{\mu}} \, \mathrm{d}\mu
=
2 \sqrt{\delta_n} \to 0
\end{align*}
as desired. This verifies~\eqref{changingclt_4}. The conclusion is that
\begin{align*}
\Big((na_n^d)^{1/2}\big(\amsmathbb{N}^{(n)}_{jk}(t | x)-\amsmathbb{E}[\amsmathbb{N}^{(n)}_{jk}(t | x)]\big)\Big)_{\!t \in [0,\theta_x]}
\end{align*}
is asymptotically equicontinuous. To transfer this result to the empirical process of interest, we need the resulting bias term to be negligible. This follows from Assumption~\ref{ass:band_kernel_add}(a), see also the first part of the proof of Proposition~\ref{prop:fidis1}.

\end{proof}

\begin{proof}[Proof of Theorem~\ref{thm:asnorm}]
By the almost sure representation theorem, we may assume that, almost surely on the supremum norm on $[0,{\theta_x}]\times[0,{\theta_x}]$,
\begin{align*}
(\xi_{jk}^{(n)}(t|x),\xi_{j}^{(n)}(t|x))_{t\in[0,{\theta_x}]}
\to
(\xi_{jk}(t|x),\xi_j(t|x))_{t\in[0,{\theta_x}]}.
\end{align*}
Note that
\begin{align*}
&\zeta_{jk}^{(n,\varepsilon)}(t|x)\\
 &= (na_n^{d})^{1/2}\int_0^t \bigg( \frac{1}{\amsmathbb{I}^{(n)}_j(s-|x) \vee \varepsilon} - \frac{1}{p_j^{\texttt{c}}(s-|x) \vee \varepsilon}\bigg) \, \amsmathbb{N}_{jk}^{(n)}(\mathrm{d}s | x) \\
 & \quad + (na_n^{d})^{1/2}\int_0^t \frac{1}{p_j^{\texttt{c}}(s-|x) \vee \varepsilon} \Big( \amsmathbb{N}_{jk}^{(n)}(\mathrm{d}s | x) 
- p_{jk}^{\texttt{c}}(\mathrm{d}s|x)\Big)\\
 &=-\big(A_n(t)+B_n(t)\big) +  C_n(t),
\end{align*}
where, almost surely on the supremum norm on $[0,{\theta_x}]$,
\begin{align*}
 A_n(t) &:= \int_0^t \frac{(na_n^{d})^{1/2}\Big(\amsmathbb{I}^{(n)}_{j}(s-|x)\vee \varepsilon-p_{j}^{\texttt{c}}(s- | x)\vee \varepsilon\Big)}{p_j^{\texttt{c}}(s-|x) \vee \varepsilon} \, \Lambda^{(\varepsilon)}_{jk}(\mathrm{d}s | x) \\
&\to\int_0^t \frac{B^{(\varepsilon)}_j(s- | x)\xi_j(s-|x)}{p_j^{\texttt{c}}(s-|x)} \, \Lambda^{(\varepsilon)}_{jk}(\mathrm{d}s | x) =: A(t), \\
 B_n(t) &:= \int_0^t \frac{(na_n^{d})^{1/2}\Big(\amsmathbb{I}^{(n)}_{j}(s-|x)\vee \varepsilon-p_{j}^{\texttt{c}}(s- | x)\vee \varepsilon\Big)}{p_j^{\texttt{c}}(s-|x) \vee \varepsilon} \Big(\mathbb{\Lambda}^{(n,\varepsilon)}_{jk}(\mathrm{d}s|x)- \Lambda^{(\varepsilon)}_{jk}(\mathrm{d}s | x)\Big)\to 0, \\
 C_n(t)&:=\int_0^t \frac{1}{p_j^{\texttt{c}}(s-|x) \vee \varepsilon} \, \xi^{(n)}_{jk}(\mathrm{d}s | x)\to \int_0^t \frac{1}{p_j^{\texttt{c}}(s-|x) \vee \varepsilon} \,\xi_{jk}(\mathrm{d}s | x) =: C(t).
\end{align*}
Thus we get the limiting process $\zeta_{jk}^{(\varepsilon)}(t|x)=-A(t)+C(t)$ as desired.

The covariances are calculated by noticing that
\begin{align*}
\zeta_{jk}^{(\varepsilon)}(t|x)=C^1(t)-C^2(t)-A(t),
\end{align*}
where
\begin{align*}
 C^1(t):= \frac{\xi_{jk}(t | x)}{p_j^{\texttt{c}}(t|x) \vee \varepsilon},\quad C^2(t):= \int_0^t \xi_{jk}(s | x) \, \mathrm{d}\Big(\frac{1}{p_j^{\texttt{c}}(s|x) \vee \varepsilon}\Big),
\end{align*}
and calculating all the cross terms using Fubini's theorem and as follows:
\begin{align*}
\phi(x)^{-1}\text{Cov}&(\zeta_{jk}^{(\varepsilon)}(t|x),\zeta_{jk}^{(\varepsilon)}(s|x)) =\frac{p^{\texttt{c}}_{jk}(s,t|x)}{(p_j^{\texttt{c}}(t|x) \vee \varepsilon)(p_j^{\texttt{c}}(s|x) \vee \varepsilon)}\\
&+ \int_0^t\int_0^s p^{\texttt{c}}_{jk}(u,v|x) \, \mathrm{d}\Big(\frac{1}{p_j^{\texttt{c}}(u|x) \vee \varepsilon}\Big)\mathrm{d}\Big(\frac{1}{p_j^{\texttt{c}}(v|x) \vee \varepsilon}\Big)\\
&+\int_0^t \int_0^s \frac{B^{(\varepsilon)}_j(u- | x)B^{(\varepsilon)}_j(v- | x) p^{\texttt{c}}_{j}(u,v|x)}{(p_j^{\texttt{c}}(u-|x) \vee \varepsilon)(p_j^{\texttt{c}}(v-|x) \vee \varepsilon)} \, \Lambda^{(\varepsilon)}_{jk}(\mathrm{d}u | x)\Lambda^{(\varepsilon)}_{jk}(\mathrm{d}v | x)\\
&- \frac{1}{p_j^{\texttt{c}}(t|x) \vee \varepsilon}  \int_0^s p^{\texttt{c}}_{jk}(t,u|x) \, \mathrm{d}\Big(\frac{1}{p_j^{\texttt{c}}(u|x) \vee \varepsilon}\Big)\\
&- \frac{1}{p_j^{\texttt{c}}(t|x) \vee \varepsilon} \int_0^s \frac{B^{(\varepsilon)}_j(u- | x) p^{3,\texttt{c}}_{jk}(t,u|x)}{p_j^{\texttt{c}}(u-|x) \vee \varepsilon} \, \Lambda^{(\varepsilon)}_{jk}(\mathrm{d}u | x) \\
&+\int_0^t \int_0^s   \frac{B^{(\varepsilon)}_j(u- | x)p^{3,\texttt{c}}_{jk}(v,u|x)}{p_j^{\texttt{c}}(u-|x) \vee \varepsilon} \, \Lambda^{(\varepsilon)}_{jk}(\mathrm{d}u | x)\,\mathrm{d}\Big(\frac{1}{p_j^{\texttt{c}}(v|x) \vee \varepsilon}\Big) \\
&- \frac{1}{p_j^{\texttt{c}}(s|x) \vee \varepsilon}  \int_0^t p^{\texttt{c}}_{jk}(s,u|x) \, \mathrm{d}\Big(\frac{1}{p_j^{\texttt{c}}(u|x) \vee \varepsilon}\Big)\\
&- \frac{1}{p_j^{\texttt{c}}(s|x) \vee \varepsilon} \int_0^t \frac{B^{(\varepsilon)}_j(u- | x)p^{3,\texttt{c}}_{jk}(s,u|x)}{p_j^{\texttt{c}}(u-|x) \vee \varepsilon} \, \Lambda^{(\varepsilon)}_{jk}(\mathrm{d}u | x) \\
&+\int_0^s \int_0^t   \frac{{B^{(\varepsilon)}_j(u- | x)}p^{3,\texttt{c}}_{jk}(v,u|x)}{p_j^{\texttt{c}}(u-|x) \vee \varepsilon} \, \Lambda^{(\varepsilon)}_{jk}(\mathrm{d}u | x)\,\mathrm{d}\Big(\frac{1}{p_j^{\texttt{c}}(v|x) \vee \varepsilon}\Big).
\end{align*}
This completes the proof.
\end{proof}

\begin{proof}[Proof of Theorem~\ref{thm:asnorm2}]
By virtue of Theorem~\ref{thm:asnorm}, and the well-known fact that the functional
\begin{align*}
z\mapsto p\Prodi_0^\cdot \big(\text{Id} + \mathrm{d}z\big)
\end{align*}
has Hadamard derivative given by
\begin{align*}
h\mapsto p\int_0^\cdot\Prodi_0^s \big(\text{Id} + \mathrm{d}z \big) \mathrm{d}h 	\Prodi_s^\cdot \big(\text{Id} + \mathrm{d}z \big),
\end{align*}
we may invoke the Functional Delta Method, see Theorem~20.8 in~\cite{Vaart1998}, to obtain that, simultaneously for all $t$,
\begin{align*}
&\gamma^{(n,\varepsilon)}(t|x)
\stackrel{\mathcal{D}}{\to}p^{(\varepsilon)}(0 | x)\int_0^t\Prodi_0^s \big(\text{Id} +  \Lambda^{(\varepsilon)}(\mathrm{d}u | x) \big) 
\zeta^{(\varepsilon)}(\mathrm{d}s|x)
\Prodi_s^t \big(\text{Id} +  \Lambda^{(\varepsilon)}(\mathrm{d}u | x) \big)
\end{align*}
as desired.
\end{proof}

\section{Bracketing central limit theorem}\label{ap:technical}

In the following, the set of all càdlàg functions $z : [a,b]\to \amsmathbb{R}$ for $a,b\in \bar{\amsmathbb{R}}$ is denoted $\amsmathbb{D}[a,b]$, while the corresponding set of all continuous and bounded functions is denoted $\amsmathbb{C}([a,b])$. The results below follow from combining results from~\cite{billingsley2013convergence} and~\cite{VaartWellner1996}.

\begin{theorem}[Asymptotic equicontinuity with uniform modulus of continuity]\label{equicon}
Assume that a sequence of $\amsmathbb{D}[a,b]$-valued stochastic processes $(\amsmathbb{X}_n)$ has converging finite-dimensional distributions. If moreover, for any $\varepsilon>0$,
\begin{align*}
\lim_{\delta\to0} \limsup_{n\to\infty}\amsmathbb{P}\bigg(\sup_{\substack{0\le s\le t \le 1 \\ |s-t|\le \delta}}|\amsmathbb{X}_n(t)-\amsmathbb{X}_n(s)|>\varepsilon\bigg)=0,
\end{align*}
then there exists a stochastic process $\amsmathbb{X}$ with $\amsmathbb{P}(\amsmathbb{X}\in \amsmathbb{C}[a,b]) = 1$ such that $\amsmathbb{X}_n\stackrel{\mathcal{D}}{\to}\amsmathbb{X}$ in $\amsmathbb{D}[a,b]$ endowed with the $J_1$-topology.
\end{theorem}

\begin{theorem}[Bracketing {central limit theorem} for separable processes]\label{clt_changingfunc} Let $(Z_{i})$ be a sequence of i.i.d. separable stochastic processes. For each $n$, let $\mathcal{F}_{n}=\{f_{n,t}:t\in T\}$ be a class of measurable functions into $\amsmathbb{R}$ indexed by a totally bounded semimetric space $(T,\rho)$. Define for every $n$ the bracketing number $N_{[\:]}(\varepsilon ||{F_{n}(Z_1)}||_{L_{2}}, \mathcal{F}_{n}, L_{2}(P))$ as the minimal number of sets $N_{\varepsilon}$ in a partition $\mathcal{F}_{n}=\bigcup_{j=1}^{N_{\varepsilon}}F_{\varepsilon_{j}}^{n}$ of the index set into sets $\mathcal{F}_{\varepsilon_{j}}^{n}$ such that for every one of these sets it holds that   
\begin{align*}
    \sum_{i=1}^{n}\amsmathbb{E}\bigg[\sup_{f,g\in \mathcal{F}^{n}_{\varepsilon_{j}}}|f_{n,t}(Z_{i})-f_{n,s}(Z_{i})|^{2}\bigg]\leq \varepsilon^{2}||{F_{n}(Z_1)}||_{L_{2}}^{2}, 
\end{align*}
where $F_{n}$ is a measurable envelope function of $\mathcal{F}_{n}$, that is a measurable function for which it holds that $|f_{n,t}|\leq F_{n}$ for all ${f_{n,t}}\in \mathcal{F}_{n}$.
Suppose that
\begin{align}
    &\amsmathbb{E}\left[{F_{n}(Z_1)}^{2}\right]=O(1) \quad \text{and} \quad
    \amsmathbb{E}\left[F_{n}^{2}1_{({F_{n}(Z_1)}>\eta \sqrt{n})}\right]\to 0 \quad \text{for every }\eta >0\label{changingclt_12},
    \\
    &\sup_{\rho(s,t)\leq \delta_{n}}\amsmathbb{E}\left[\left({f_{n,t}(Z_1)}-{f_{n,s}(Z_1)}\right)^{2}\right]\to 0 \quad \text{for every }\delta_{n}\to 0\label{changingclt_3},
    \\
    &\int_{[0,\delta_{n}]}\sqrt{\log\left(N_{[\:]}(\varepsilon||{F_{n}(Z_1)}||_{L_{2}}, \mathcal{F}_{n}, L_{2}(\amsmathbb{P}))\right)}\mathrm{d}\varepsilon \to 0 \quad\text{for every }\delta_{n}\to 0 \label{changingclt_4}.
\end{align}
Then the sequence 
\begin{align*}
\Big(n^{-1/2}\sum_{i=1}^n\big(f_{n,t}({Z_i})-\amsmathbb{E}[f_{n,t}({Z_1})]\big)\Big)_{f_{n,t}\in \mathcal{F}_{n}}
\end{align*}
is asymptotically $\rho$-equicontinuous.
If the partitions can be chosen independent of $n$, then the condition of~\eqref{changingclt_3} is unnecessary. 
\end{theorem}
\begin{lemma}\label{separable_measurable}
Assume $\amsmathbb{X}$ is a separable stochastic process indexed by the set $\mathcal{F}$. Then $||\amsmathbb{X}||_\mathcal{F}=\sup\{f\in\mathcal{F} : \amsmathbb{X}(f)\}$ is measurable.
\end{lemma}
Since all càdlàg stochastic processes indexed by an interval $[a,b]\subset \amsmathbb{R}$ are separable, all processes treated in this paper are separable. Thus, by Lemma~\ref{separable_measurable} the supremum  process is measurable and may hence serve as the measurable envelope in Theorem~\ref{clt_changingfunc}. Finally, when taking $\rho$ as the Euclidean distance, we may obtain from concatenating Theorems~\ref{clt_changingfunc} and~\ref{equicon} a bracketing central limit theorem for convergence in $\amsmathbb{D}[0,\theta_x]$ endowed the $J_1$-topology. In that case, the limit is a tight Gaussian process.

\end{appendix}

\bibliographystyle{imsart-nameyear} 
\bibliography{main.bib}       

%
%
%

\newpage
\setcounter{page}{1}
\pagenumbering{roman} 

{
\section*{Supplementary Material}

The purpose of this section is to showcase the finite sample performance of the conditional Aalen--Johansen estimator in a basic model, namely the illness-death model without recovery. The main goal is to highlight how a duration effect can be targeted by conditioning on an internal covariate; a feature of our estimator that also applies for general state spaces. It should be mentioned that specifically for the illness-death model without recovery, there is a plethora of tailored solutions that may perform better, see~\cite{MunchEtAl2023} and the references therein. The strength of our method lies in its applicability to arbitrary state spaces and data generating processes, without incurring added complications. 

Concretely, we consider a time-inhomogeneous semi-Markov process on the state space $\{1,2,3\}$ with initial state $Z_0 = 1$ and non-zero transition rates given by
\begin{align*}
\lambda_{12}(t, u)
&= 0.09 + 0.0018t, \\
\lambda_{13}(t, u)
&= 0.01 + 0.0002t, \\
\lambda_{23}(t, u)
&= 0.09 + \mathds{1}_{\{u < 4\}}0.20 + 0.001t.
\end{align*}
Note that the functional form of the duration effect is crude. We simulate $1,000$ and $5,000$ independent and identically distributed realizations subject to entirely random right-censoring, respectively. Right-censoring follows the distribution $\text{Unif}(10,40)$. The degrees of censoring are $0.162$ and $0.1756$, respectively.

We first fit the unconditional Aalen--Johansen estimator. For illustrative purposes, we plot in Figure~\ref{fig1} the estimate of the state occupation probability $p_2$ (dashed) together with its true value (full).

\begin{figure}[h!]
\centering
\includegraphics[width=0.48\textwidth]{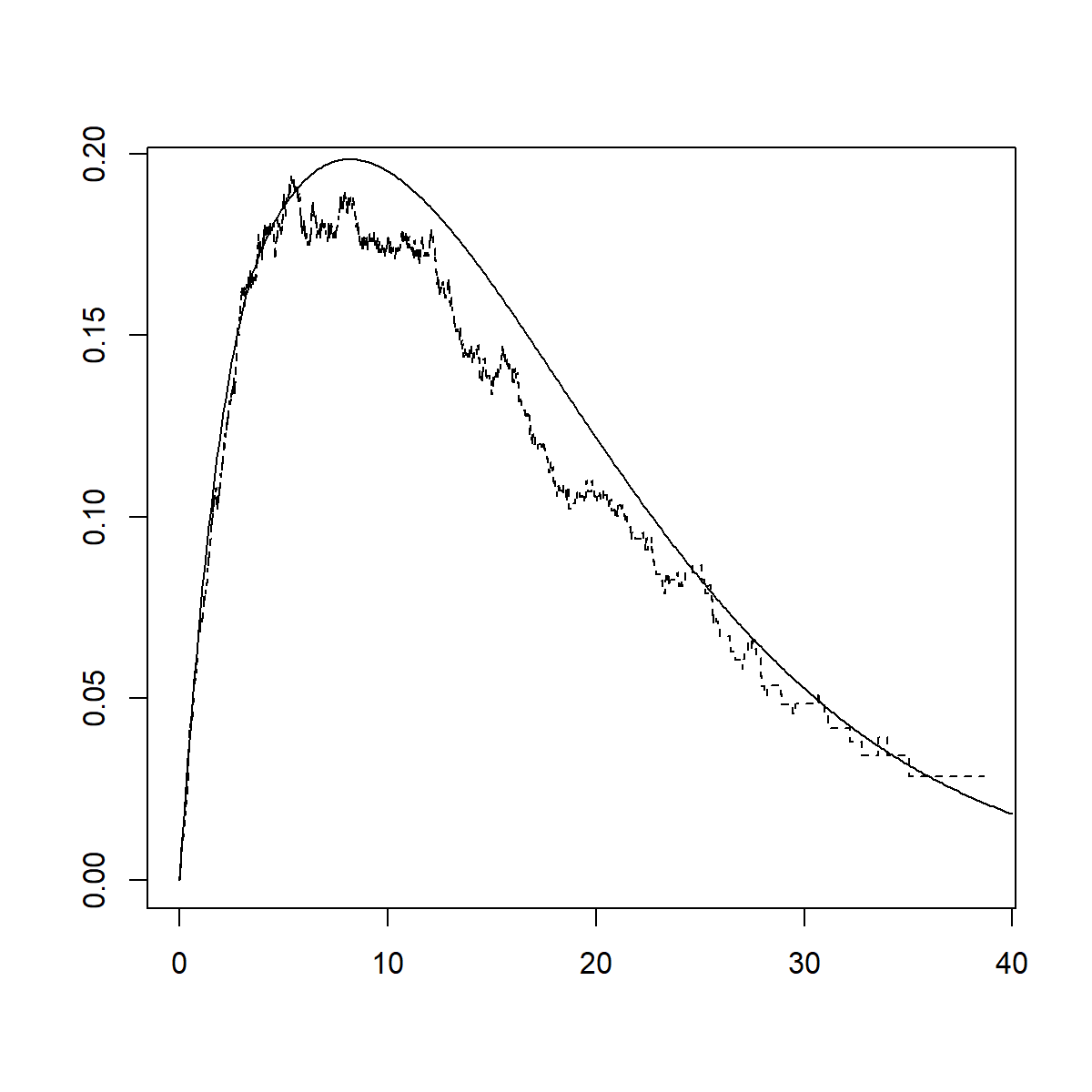}
\includegraphics[width=0.48\textwidth]{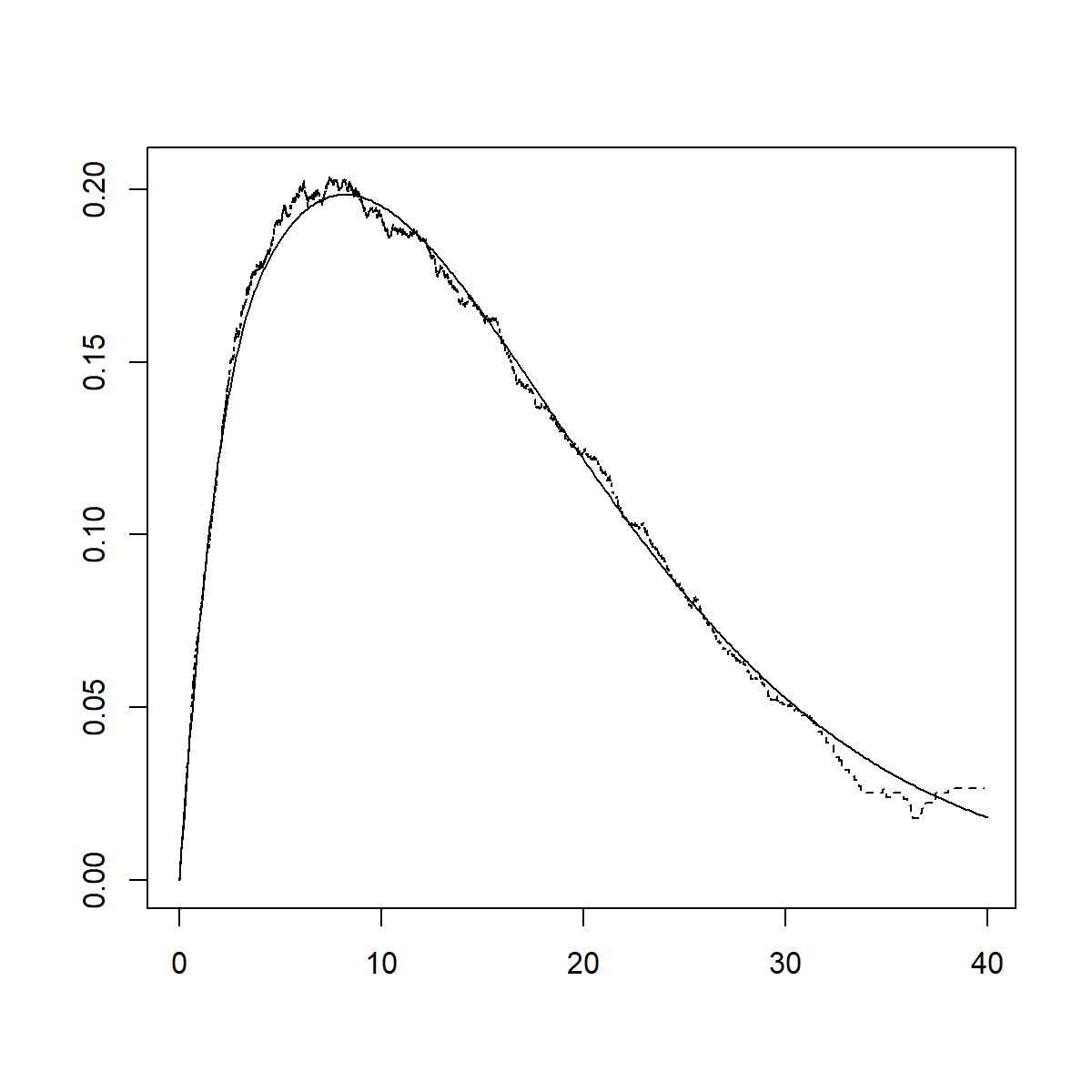}
\caption{Estimate of the state occupation probability $p_2$ (dashed) together with its true value (solid) for $n=1,000$ (left) and $n=5,000$ (right), respectively.}
\label{fig1}
\end{figure}

We next fit the ordinary Aalen--Johansen estimator as well as the conditional Aalen--Johansen estimator, conditioning on the process being in state $2$ at time $10$ (landmarking). This corresponds to estimates of the transition probabilities (conditional occupation probabilities)
\begin{align*}
t \mapsto \amsmathbb{P}(Z_t = j \, | \, Z_{10} = 2).
\end{align*}
The effective sample sizes are $176$ of $1,000$ and $962$ of $5,000$, respectively. We plot in Figure~\ref{fig2} the resulting estimates (dashed) together with the true value (solid) for $j=2$.
\begin{figure}[h!]
\centering
\includegraphics[width=0.48\textwidth]{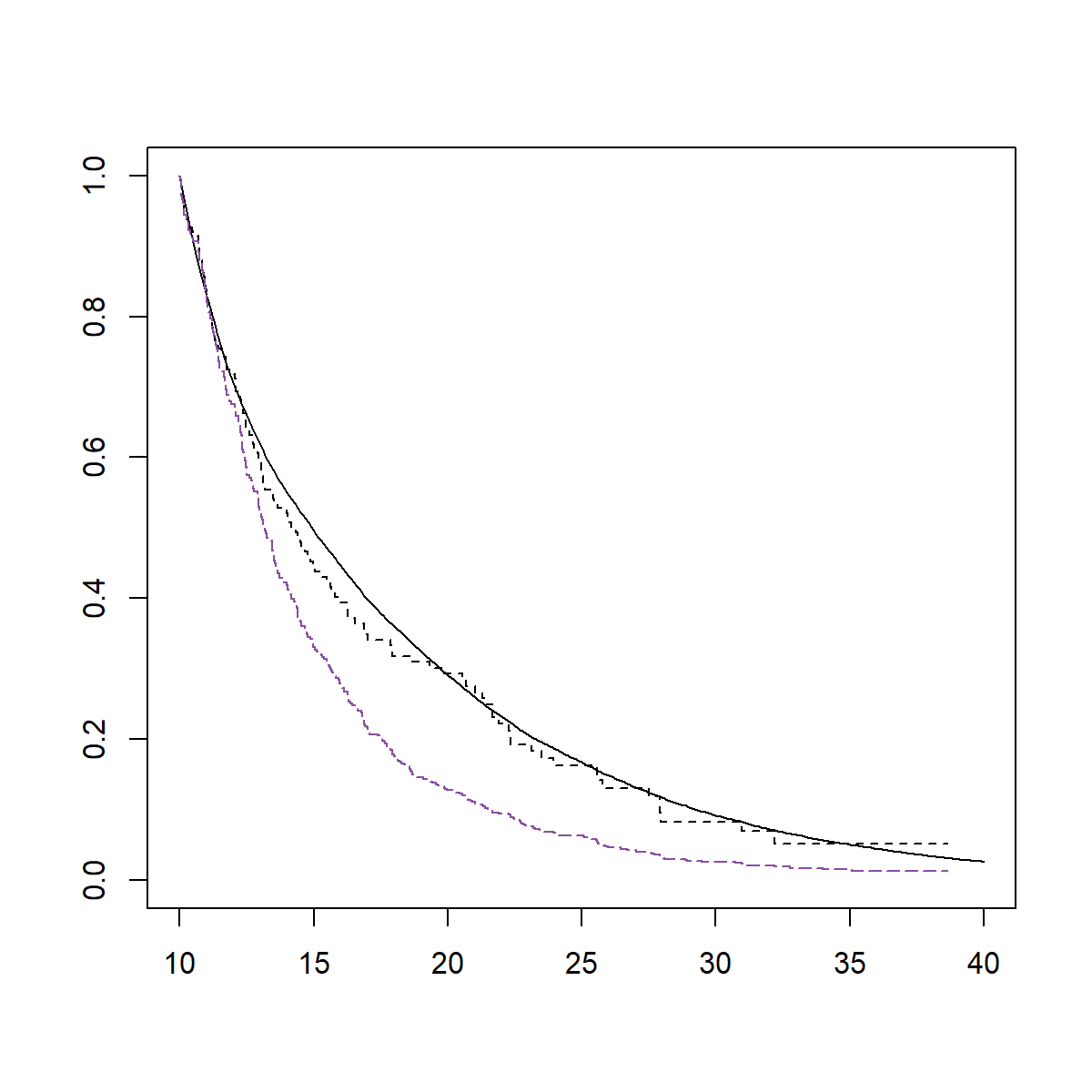}
\includegraphics[width=0.48\textwidth]{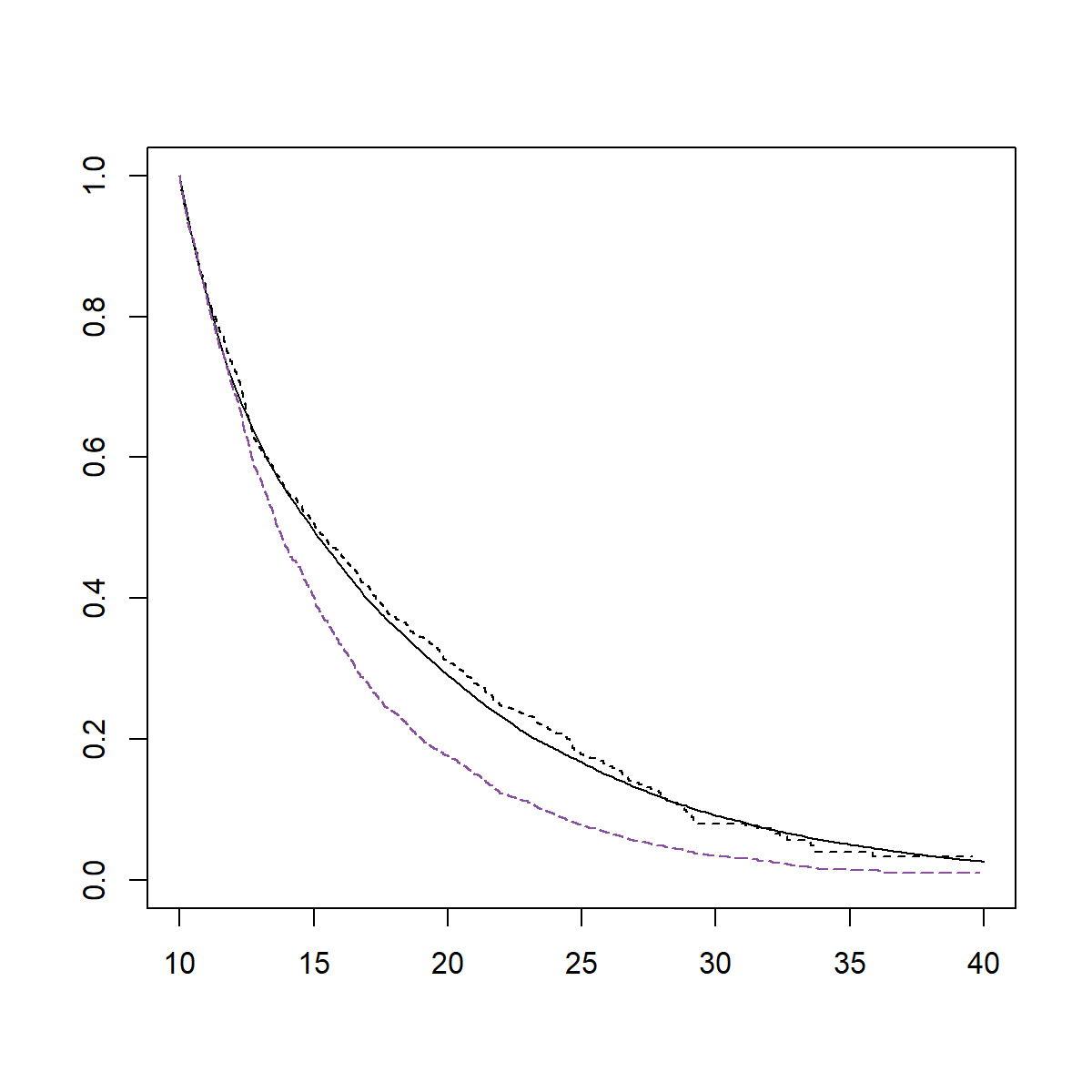}
\caption{Estimates of the transition probability with $j=2$ (dashed) using the ordinary Aalen--Johansen estimator (magenta) and the conditional Aalen--Johansen estimator (black) together with the true value (solid) and for $n=1,000$ (left) and $n=5,000$ (right), respectively.}
\label{fig2}
\end{figure}

Finally, we fit the conditional Aalen--Johansen estimator, conditioning on the process having at time $10$ been in state $2$ for $u=1,5$ units of time, which is continuous internal covariate. We use a uniform kernel. This corresponds to kernel estimates of the transition probabilities
\begin{align*}
t \mapsto \amsmathbb{P}(Z_t = j \, | \, Z_{10} = 2, U_{10} = u),
\end{align*}
where $(U_t)$ is the duration process associated with $Z$. The effective sample sizes for $u=1$ are $56$ of $1,000$ and $166$ of $5,000$, respectively, while the effective samples sizes for $u=5$ are $28$ of $1,000$ and $93$ of $5,000$, respectively. Bandwidth selection was performed according to the marginal distribution of the internal covariate, $U_{10}$, implementing the method of~\cite{SheatherJones1991}, which is also the standard of the R package AalenJohansen, see~\cite{BladtFurrer2023}. In Figure~\ref{fig3}, we plot the resulting estimates (dashed) and true values (solid) for $j=2$. We also include the former landmark estimator for reference.

\begin{figure}[h!]
\centering
\includegraphics[width=0.48\textwidth]{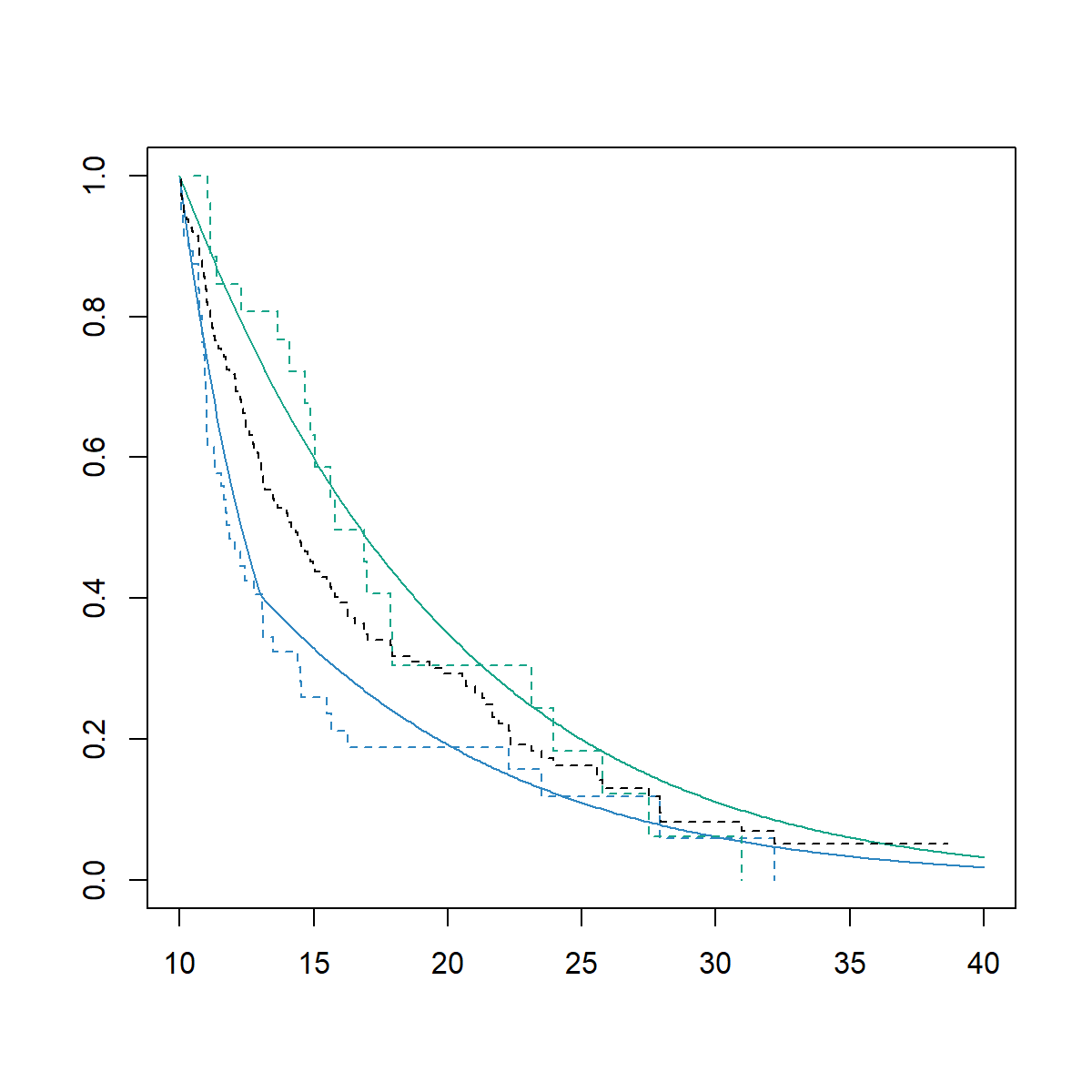}
\includegraphics[width=0.48\textwidth]{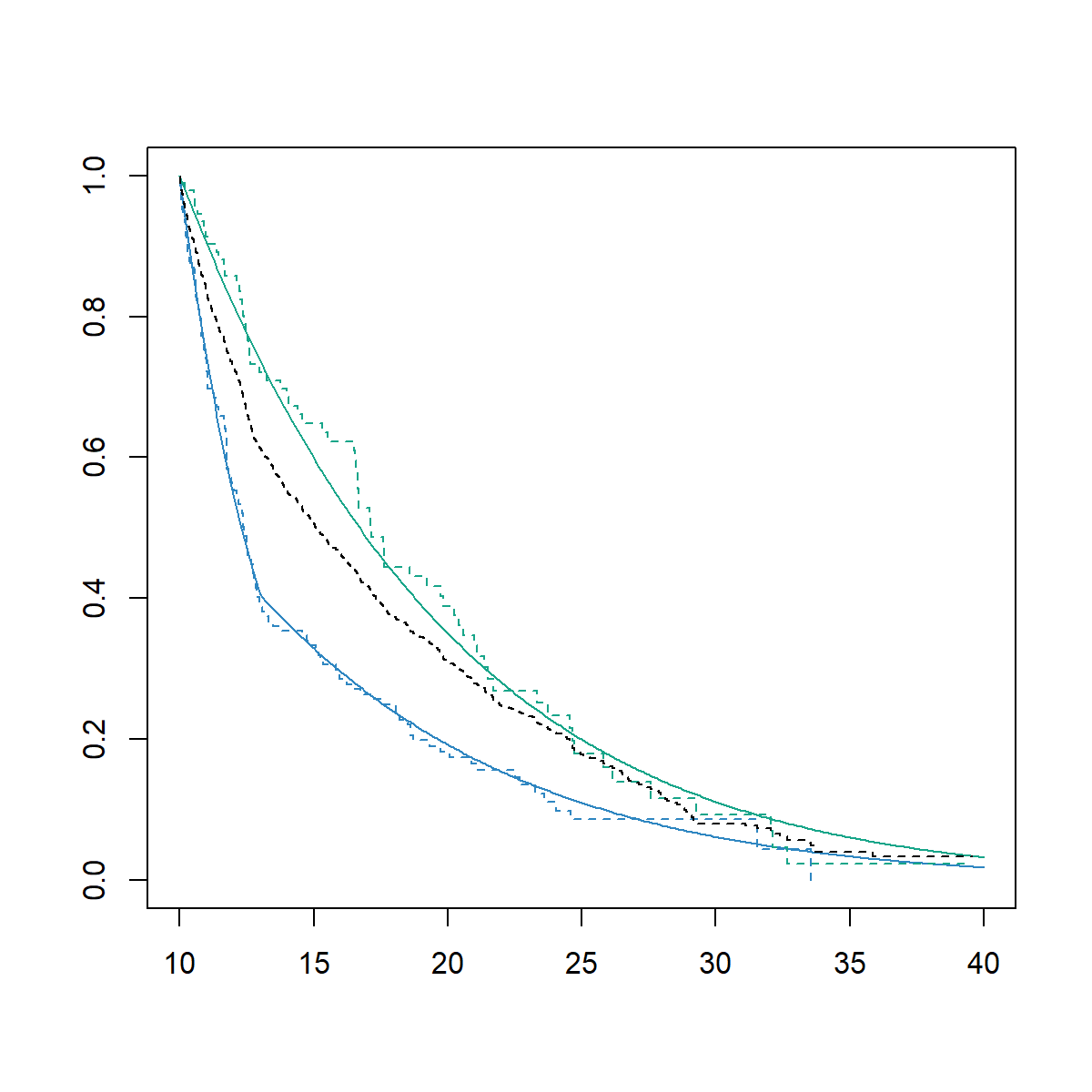}
\caption{Estimates of the transition probability with $j=2$ (dashed) for $u=1$ (blue), $u=5$ (green), and without conditioning on duration (black) for $n=1,000$ (left) and $n=5,000$ (right), respectively. True values are in solid.}
\label{fig3}
\end{figure}

An obvious comment is that with increasing accuracy, that is with additional conditioning, the volatility of the estimator increases due to reduced effective sample size. In general, we observe that we are far away from the asymptotic case. Consequently, it is of great importance to understand, based on the application, which type of transition probability (conditional occupation probability) is required. The conditional Aalen--Johansen estimator allows for the inclusion of non-categorical conditioning variables, such as the current duration. Its greatest strength lies in its capacity to unveil non-Markov behavior, such as identifying a clear duration effect, compare with Figure~\ref{fig3}, but also further non-Markov effects. Beyond its relevance as an exploratory tool, its potential for statistical inference in real world applications ultimately depends on the bias-variance trade-off at hand, which is primarily driven by the size of the sample.}

\end{document}